\pdfoutput=1
\documentclass[10pt,a4paper,oneside]{amsart}
\usepackage{amsmath, amssymb, amsfonts}
\usepackage[a4paper]{geometry}

\usepackage{comment}

\usepackage{mathrsfs}

\numberwithin{equation}{section}

\DeclareMathOperator{\coker}{coker}
\DeclareMathOperator{\Col}{Col}
\DeclareMathOperator{\Colu}{\underline{Col}}

\DeclareMathOperator{\End}{End}

\DeclareMathOperator{\Gal}{Gal}

\DeclareMathOperator{\Hom}{Hom}
\DeclareMathOperator{\id}{id}

\DeclareMathOperator{\Norm}{\mathfrak N}
\DeclareMathOperator{\ord}{ord}

\DeclareMathOperator{\rank}{rank}

\DeclareMathOperator{\Sel}{Sel}
\DeclareMathOperator{\Tr}{Tr}

\renewcommand{\aa}{\mathfrak a}
\newcommand{\Qn}{\QQ_{(n)}}
\newcommand{\F}{\mathcal F}
\newcommand{\Qcyc}{\QQ_{\cyc}}
\newcommand{\Qcycp}{\QQ_{\cyc,p}}
\newcommand{\FF}{\mathbb F}
\newcommand{\G}{\mathcal G}
\newcommand{\II}{\mathbb I}
\newcommand{\KK}{\mathcal K}
\newcommand{\LL}{\mathcal L}
\newcommand{\mm}{\mathfrak m}
\newcommand{\OO}{\mathcal O}

\newcommand{\QQ}{\mathbb Q}
\newcommand{\T}{\mathcal T}
\newcommand{\ZZ}{\mathbb Z}
\newcommand{\cyc}{\mathrm{cyc}}

\renewcommand{\epsilon}{\varepsilon}
\renewcommand{\theta}{\vartheta}
\newcommand{\ab}{\mathrm{ab}}

\newcommand{\Ehat}{\widehat E}
\newcommand{\Iw}{\mathrm{Iw}}
\newcommand{\omegat}{\widetilde\omega}

\newcommand{\sK}{\mathscr{K}_\infty}
\newcommand{\tsK}{{\mathscr{K}}_\infty}

\newcommand{\Zp}{\ZZ_p}
\newcommand{\Qp}{\QQ_p}
\newcommand{\Qpt}{\QQ_{p^2}}

\newcommand{\f}{\mathrm{f}}
\newcommand{\cC}{\mathcal{C}}
\newcommand{\cE}{\mathcal{E}}
\newcommand{\cU}{\mathcal{U}}

\newcommand{\cX}{\mathcal{X}}
\newcommand{\sL}{\mathscr{L}}
\newcommand{\sM}{\mathscr{M}}
\newcommand{\cA}{\mathcal{A}}
\newcommand{\cV}{\mathcal{V}}
\newcommand{\bz}{\mathbf{z}}
\newcommand{\length}{\mathrm{length}}
\usepackage{upgreek}
\usepackage[OT2,T1]{fontenc}
\DeclareSymbolFont{cyrletters}{OT2}{wncyr}{m}{n}
\DeclareMathSymbol{\Sha}{\mathalpha}{cyrletters}{"58}

\newcommand{\cY}{\mathcal{Y}}

\usepackage{colonequals}
\usepackage{stmaryrd}
\usepackage{extarrows}

\NewDocumentEnvironment{noproof}{m}{
  \par\pushQED{\qed}\UseName{#1}%
}{\popQED\UseName{end#1}}

\usepackage{tikz}
\usetikzlibrary{cd}
\usetikzlibrary{arrows}

\usepackage[style=alphabetic,backend=biber,isbn=false,url=false,maxbibnames=99,maxcitenames=99,maxalphanames=4,minalphanames=4]{biblatex}
\usepackage{enumitem}

\usepackage[hidelinks,plainpages=false,pdfpagelabels]{hyperref}
\usepackage{thmtools} 
\usepackage[capitalise, noabbrev]{cleveref}

\usepackage{xcolor}

\hypersetup{
 colorlinks=true,
 linkcolor=blue,
 filecolor=blue,
 citecolor=olive,
 urlcolor=orange,
 }

\definecolor{ForestGreen}{RGB}{34,139,34}

\definecolor{DarkOrchid}{RGB}{164,83,138}

\theoremstyle{plain}
\newtheorem{theorem}{Theorem}[section]
\newtheorem{lemma}[theorem]{Lemma}
\newtheorem{proposition}[theorem]{Proposition}
\newtheorem{corollary}[theorem]{Corollary}

\theoremstyle{definition}
\newtheorem{definition}[theorem]{Definition}

\newtheorem{mainThm}{Theorem}

\theoremstyle{remark}
\newtheorem{remark}[theorem]{Remark}

\Crefname{lemma}{Lemma}{Lemmata}
\Crefname{proposition}{Proposition}{Propositions}
\Crefname{conjecture}{Conjecture}{Conjectures}
\Crefname{example}{Example}{Examples}

\title{Iwasawa theory of CM elliptic curves at potentially supersingular primes}

\author[B.~Forrás]{Ben Forrás}
\address[Forrás]{Department of Mathematics and Statistics\\University of Ottawa\\
150 Louis-Pasteur Pvt\\
Ottawa, ON\\
Canada K1N 6N5}
\email{ben.forras@uottawa.ca}

\author[A.~Lei]{Antonio Lei}
\address[Lei]{Department of Mathematics and Statistics\\University of Ottawa\\
150 Louis-Pasteur Pvt\\
Ottawa, ON\\
Canada K1N 6N5}
\email{antonio.lei@uottawa.ca}

\subjclass[2020]{
    11R23 (primary), 
    11G05, 
    11G15 
    (secondary)}
\keywords{Iwasawa theory, CM elliptic curves, potentially supersingular primes}

\begin{document}

\begin{abstract}
We develop a plus and minus Iwasawa theory for CM elliptic curves with potentially supersingular reduction at a prime $p\ge5$. We construct local points that satisfy suitable ``jumping trace'' relations using the Honda--Demchenko theory applied to height-two formal groups over local fields with small ramification degree. These local points allow us to study plus and minus Selmer groups of CM elliptic curves over the cyclotomic $\mathbb{Z}_p$-extension, to construct the corresponding plus and minus $p$-adic $L$-functions, and to prove an Iwasawa main conjecture relating these objects. As an application, we obtain asymptotic growth formulae for the $p$-primary part of the Tate--Shafarevich groups.
\end{abstract}

\maketitle

\section{Introduction}

Let $E/\QQ$ be an elliptic curve, and let $p \ge 5$ be a fixed rational prime. In Iwasawa theory one studies arithmetic objects attached to $E$ that encode the behaviour of $E$ over the cyclotomic $\Zp$-extension of $\QQ$. The first results in this direction are due to Mazur \cite{mazur72}, under the assumption that $p$ is a good ordinary prime for $E$. In particular, he established a control theorem for the $p$-primary Selmer groups over $\Zp$-extensions, and used it to deduce results on the growth of the $p$-primary part of the Shafarevich--Tate group under suitable hypotheses.

Subsequently, Kato \cite{kato04} proved that the Pontryagin dual of the Selmer group over the cyclotomic $\Zp$-extension of an abelian number field is a torsion module over the Iwasawa algebra, and verified one inclusion in the Iwasawa main conjecture. This conjecture relates the characteristic ideal of the Pontryagin dual of the Selmer group to the $p$-adic $L$-function of Mazur--Swinnerton-Dyer \cite{MSD}. When $E$ has complex multiplication, the full main conjecture was established by Rubin \cite{rubin91}. In the non-CM case, Skinner and Urban \cite{skinnerurban} proved the main conjecture under certain hypotheses on the image of the Galois representation attached to the $p$-adic Tate module of $E$.

When $E$ has good supersingular reduction at $p$, the classical Selmer group over the cyclotomic $\mathbb{Z}_p$-extension of $\mathbb{Q}$ fails to satisfy Mazur's control theorem, and its Pontryagin dual is no longer a torsion module over the Iwasawa algebra. Moreover, although the classical $p$-adic $L$-functions constructed by Amice--V\'elu \cite{amicevelu} and Vi\v{s}ik \cite{visik} interpolate the complex $L$-values of $E$ twisted by Dirichlet characters of $p$-power conductor, they do not lie in the Iwasawa algebra. Consequently, the classical formulation of the Iwasawa main conjecture does not extend to the supersingular setting in a straightforward manner.

Assuming $a_p(E) = 0$, Pollack constructed the plus and minus $p$-adic $L$-functions in \cite{pollack03}. These are elements of the Iwasawa algebra obtained by decomposing the classical $p$-adic $L$-function using certain plus and minus logarithms. On the algebraic side, Kobayashi \cite{Kobayashi} introduced the plus and minus Selmer groups, whose Pontryagin duals are torsion over the Iwasawa algebra. He formulated an Iwasawa main conjecture relating these analytic and algebraic objects, proved one direction of the conjectural equality of characteristic ideals, and deduced results on the asymptotic growth of the $p$-primary part of the Tate--Shafarevich groups. In the CM case, the full main conjecture was subsequently established by Pollack and Rubin \cite{PollackRubin}.

In \cite{KimPark}, Kim and Park studied analogous questions for CM elliptic curves with good supersingular reduction at $p$, without assuming that the curve is defined over $\QQ$. In particular, they defined plus and minus Selmer groups together with the corresponding $p$-adic $L$-functions, and proved the Iwasawa main conjecture over abelian extensions of the CM field. As in Kobayashi's work, a key ingredient is the construction of certain local points satisfying a ``jumping trace'' condition using Honda's theory on formal groups over local fields.

The main objective of this article is to develop a novel plus and minus theory for CM elliptic curves with potentially supersingular reduction at $p$. Let $E/\QQ$ be an elliptic curve with additive reduction at $p$. When $E$ acquires good supersingular reduction over a quadratic extension of $\QQ$, one can define plus and minus $p$-adic $L$-functions attached to $E$ by passing to a suitable quadratic twist, in a manner analogous to Delbourgo's work on potentially ordinary primes \cite{delbourgoComp,DelbourgoJNT}; see \cite[Theorem~4.3]{naman}. This leads to a natural definition of plus and minus Selmer groups and the formulation of corresponding Iwasawa main conjectures, in the spirit of Kobayashi \cite{Kobayashi}. A natural question then arises: what happens when $E$ attains good supersingular reduction over a non-quadratic extension of $\QQ$? In this article we address this question in the case where $E$ has complex multiplication.

Let $K$ be an imaginary quadratic field. Let $p\ge5$ be a prime number that is inert in $K$. Let $E/K$ be an elliptic curve with complex multiplication by an order in $\OO_K$. Let $\LL/\Qp$ be a finite extension where $E$ has good supersingular reduction, containing the $p$-completion of $K$. In particular, $\LL/\Qp$ is a ramified extension. In \cite{Kobayashi} and \cite{KimPark}, the starting point of the plus and minus theory is the use of Honda theory of formal groups to construct local points over the cyclotomic extension of $\Qp$ and $\QQ_{p^2}$, respectively. These points satisfy a ``jumping trace'' condition, and they play a crucial role in the definition of the Coleman maps. In our setting, this method does not apply directly due to the ramification of $\LL/\Qp$. Our new input is to apply the work of Demchenko \cite{Demchenko}, which is a generalization of Honda theory for formal groups over local fields with small ramification degree over $\Qp$, to construct local points on height-two formal groups over $\LL$. See \cref{cor:generation} for further details. We also make use of the explicit description of the formal group attached to a CM elliptic curve in the work of Sairaiji \cite{SairaijiRM}, which allows us to apply our local theory to a specific family of CM elliptic curves.

 Let $K_\cyc$ denote the cyclotomic $\Zp$-extension of $K$. Via the ``jumping trace'' conditions as in  \cite{Kobayashi,KimPark}, we define the plus and minus Selmer groups  $\Sel_{p^\infty}^\pm(E/K_\cyc)$ (see \cref{def:pmSel}). 
Using the local points we construct, we define plus and minus Coleman maps (see \S\ref{sec:Coleman-maps}), which allows us to define plus and minus $p$-adic $L$-functions $L_p^\pm(E,K_\cyc)$ (see \cref{def:padicL}) after applying these maps to the cohomological classes defined by elliptic units. These $p$-adic $L$-functions interpolate the complex $L$-values of $E$ twisted by Hecke characters that factor through $\Gal(K_\cyc/K)$. Further, we prove that they satisfy an Iwasawa main conjecture:

\begin{mainThm}[\cref{cor:IMC}]\label{thmA}
Let $K=\QQ(\sqrt{-1})$ or $K=\QQ(\sqrt{-3})$. Let $p\ge5$ be a prime number that is inert in $K$. Let $E/K$ be an elliptic curve with complex multiplication by an order in $\OO_K$. Assume that $E$ has potentially supersingular reduction at $p$, and that $K$ is contained in a certain field $L$ defined in \eqref{eq:L-def} over which $E$ has good reduction at $p$. When $p=5$, assume further that the semistability defect defined in \eqref{eq:e} is not equal to $6$.

Let $\KK$ denote the unique unramified quadratic extension of $\Qp$. Let $\Lambda^{\OO_\KK}$ be the Iwasawa algebra $\OO_\KK[[\Gal(K_\cyc/K)]]$.
    The $\Lambda^{\OO_\KK}$-module $\Hom_{\OO_{ K}}(\Sel^\pm(E/ K_\cyc),\KK/\OO_\KK)$ is torsion. Furthermore, its characteristic ideal is generated by  $ L_p^\pm(E,K_\cyc) $.
\end{mainThm}

We make use of the assumption that $K$ is contained in the field $L$ defined in \eqref{eq:L-def} when invoking the results of Sairaiji \cite{SairaijiRM} in order to guarantee that the local theory we develop can indeed be applied. In particular, this assumption forces $K$ to be either $\QQ(\sqrt{-1})$ or $\QQ(\sqrt{-3})$. After the local theory has been set up, our proof of \cref{thmA} proceeds by closely following the approach of Pollack–Rubin in \cite{PollackRubin}, where they establish the Iwasawa main conjecture for elliptic curves over $\QQ$ with good supersingular reduction.
Furthermore, following the work of Kobayashi in \cite{Kobayashi}, we show that the growth of the $p$-part of the Tate--Shafarevich groups inside $K_\cyc$ is governed by the Iwasawa invariants of the signed Selmer groups together with the contribution from certain cyclotomic polynomials and the limit of the Mordell--Weil ranks.

\begin{mainThm}[\cref{cor:sha}]\label{thmB}
Let $K$, $p$ and $E$ be as in Theorem~\ref{thmA}.
    Assume that $\Sha(E/K_n)[p^\infty]$ is finite with cardinality $p^{e_n}$ for all $n\ge0$. Let $2r_\infty=\displaystyle\lim_{n\to\infty}\rank E(K_n)$. Let $\mu^\pm$ and $\lambda^\pm$ be the Iwasawa invariants of  $\Sel^\pm(E/ K_\cyc)^\vee$ (see \eqref{eq:def-lambda-mu-pm} where we recall their definitions).  For $n\gg0$, we have
\[
e_n-e_{n-1}= \begin{cases}
      2(\deg\omegat_n^-+\lambda^+-r_{\infty}+\varphi(p^n)\mu^+)  & \text{if $n$ is even,}\\
      2(\deg\omegat_n^++\lambda^--r_{\infty}+\varphi(p^n)\mu^-)  & \text{if $n$ is odd,}
    \end{cases}
\]
where $\omegat_n^\pm=\prod_{i} \Phi_i$ is the product of the $p^i$th cyclotomic polynomials, with the product being taken over $i$ from $1$ to $n$ such that $(-1)^i=\pm1$.
\end{mainThm}

\cref{thmA} and \cref{thmB} also have consequences for the Iwasawa of $E$ over the cyclotomic $\Zp$-extension $\Qcyc$ of $\QQ$. By descending from $K$ to $\QQ$, we show in \cref{thm:IMC-Q} that the plus and minus Selmer groups of $E$ over $\Qcyc$ satisfy the Iwasawa main conjecture in a way analogous to the assertions of \cref{thmA}. 

 \cref{cor:sha-Q} shows that the growth of the $p$-primary part of the Tate--Shafarevich groups along $\Qcyc$ is exactly half of the corresponding growth of $\Sha(E/K_n)[p^\infty]$ described in Theorem~\ref{thmB}. This gives a partial answer to Conjecture 4.11 stated in \cite{naman} and provides a theoretical explanation for the numerical calculations of the $\lambda$-invariants of Mazur–Tate elements carried out in \textit{op. cit.}

It is worth emphasizing that our approach differs significantly from the one employed in \cite{delbourgoComp,DelbourgoJNT,MuellerGlasgow}, where the Iwasawa theory of elliptic curves and CM modular forms at a prime of potentially ordinary reduction is studied.
In those works, one treats bad reduction by twisting (and later untwisting) by an appropriately chosen Dirichlet character, which is not always feasible in the potentially supersingular case due to the structural difference between the Galois representations.
Moreover, in the construction of $p$-adic $L$-functions in the aforementioned works, one uses a certain Coleman map from $p$-adic Hodge theory, which is currently not available in the potentially supersingular setting unless good reduction is achieved over a quadratic extension.
This divergence arises because potentially ordinary reduction becomes good after a finite abelian (hence cyclotomic) extension of $\Qp$, whereas in the potentially supersingular case the analogous extension required to achieve good reduction may be non-abelian.

\subsection*{Outlook}
Our construction of local points in \S\ref{sec:points} is carried out for a particular family of height-two formal groups defined over a suitably chosen local field. This method seems to be amenable to more general situations. For the purposes of our applications to elliptic curves, we restrict our attention to CM elliptic curves, since we rely on Sairaiji’s work \cite{SairaijiRM}, which provides an explicit description of the associated formal group and thus enables us to check that our local point construction is applicable. One may expect that an analogous approach should also work for non-CM elliptic curves at the expense of imposing extra hypotheses. 

The condition that $E$ has potentially supersingular reduction at primes above $p$ forces $p$ to be non-split in the CM field $K$ by the Shimura–Taniyama formula. The method developed in this article to construct local points depends crucially on the assumption that $p$ is inert in $K$. When $p$ is ramified in $K$, the numerical evidence collected in \cite{naman} indicates that there exist elliptic curves for which the growth of the Tate–Shafarevich groups behaves in a manner quite different from that predicted by Theorem~\ref{thmB}. This suggests that a straightforward extension of our construction to the ramified setting is unlikely, and that one may instead need to work with a different class of Selmer groups. In the anticyclotomic setting, recent breakthrough results on Iwasawa main conjectures and on the asymptotic behaviour of Mordell–Weil ranks have been established in \cite{BKO1,BKO2,BKO3,BKNO1,BKNO2}, for any prime $p$ that is non-split in $K$. It would be of interest to investigate the corresponding phenomena in the cyclotomic extension.

\subsection*{Organisation}
We first develop the local theory.
In \S\ref{sec:relaxed-Honda}, we discuss mild generalizations of Honda theory in the case of ramified extensions.
We then construct a suitable formal group over a certain tamely ramified extension of $\Qp$, and describe a system of universal norms in \S\ref{sec:points}.
Turning to the global picture, in \S\ref{sec:CM-formal}, we use the main theorem of complex multiplication and Honda theory to identify the formal groups of our elliptic curves with the ones just constructed. Then in \S\ref{sec:IMC} we utilize the theory of elliptic units to obtain plus and minus $p$-adic $L$-functions as well as an Iwasawa main conjecture. Finally in \S\ref{sec:growth-of-Sha}, we discuss the growth of Tate--Shafarevich groups.

\subsection*{Acknowledgements} We thank Daniel Delbourgo, Byoung Du Kim, Katharina M\"uller and Naman Pratap for interesting discussions during the preparation of this article. We also thank Oleg Demchenko for kindly making the English translation of his article \cite{Demchenko} available to us.  BF's research is funded by the Deutsche Forschungsgemeinschaft (DFG, German Research Foundation), project no.~559516518. AL's research is supported by the NSERC Discovery Grant RGPIN-2026-04351.

\section{generalizations of Honda theory} \label{sec:relaxed-Honda}
Classically, the main results of Honda theory \cite{HondaCommFormalGroups} hold over unramified extensions of $\Qp$. However, as we shall see in \S\ref{sec:setup}, the potentially supersingular situation requires us to work over tamely ramified extensions of degree dividing $p+1$. In this section, we describe how to adapt Honda theory to this setup.

\subsection{Relaxed Honda types}
Let $\LL/\QQ_p$ be a finite extension, let $\pi\in\OO_\LL$ be a uniformizer, and let $q$ be the order of the residue field $\OO_\LL/\pi\OO_\LL$. Let $\LL[[X]]_0=X\LL[[X]]$ denote the ring of power series in one variable with zero constant term. We define $\OO_\LL[[X]]_0$ similarly.

We extend Honda's notion of formal group types given in \cite[\S2]{HondaCommFormalGroups}. Our setting is significantly simpler: we only work with one-dimensional formal groups in a single variable over finite extensions of $\QQ_p$. In this case, Honda's condition $(F)$, which asserts that there is an endomorphism $\sigma$ of $\LL$ and an integer $p^r$ such that $\alpha^\sigma\equiv\alpha^{p^r}\mod \pi\OO_\LL$ for all $\alpha\in \OO_\LL$,  holds with the endomorphism being the identity and $p^r=q$. Unlike \cite{HondaCommFormalGroups}, where a uniformizer of $\LL$ is fixed in the definition, we allow the choice of uniformizers of a specified intermediate extension.
\begin{definition}
     Let $\LL'$ be an intermediate extension of $\LL/\QQ_p$.
     Let $f\in\LL[[X]]_0$ be a power series and $u=\sum_{i=0}^\infty u_nX^n\in\OO_\LL[[X]]$ such that $u(0)$ is a uniformizer in $\LL'$. 
     Let $u*f\in\LL[[X]]_0$ be the power series given by
     \[(u*f)(X)\colonequals\sum_{n=0}^\infty u_n f(X^{q^n}).\]
     We say that $f$ is of \emph{relaxed Honda type} $u$ if $f(X)\equiv X \pmod{X^2\LL[[X]]}$ and $(u*f)(X)\equiv 0\pmod{\pi\OO_\LL[[X]]}$.
\end{definition}
The difference between this and Honda's formulation is that one works with two possibly distinct fields, and the congruence involving $u*f$ is not modulo $u(0)$. Note that if $\LL=\LL'$, the relaxed Honda type recovers the usual definition of Honda types for one-dimensional formal groups.

For the convenience of the reader, we state the modified versions of the relevant results from \cite{HondaCommFormalGroups}; the proofs are identical to their non-relaxed counterparts, with the uniformizer $\pi$ in Honda's work replaced by $u_0$, and the prime ideal $\mathfrak p$ replaced by $\pi\OO_\LL$.

\begin{lemma}[{\cite[Lemma~2.2]{HondaCommFormalGroups}}] \label{Honda2.2}
    Let $u(X)=\sum_{n=0}^\infty u_n X^n\in\OO_\LL[[X]]$ be a power series such that $u_0$ is a uniformizer in $\LL'$, and write $u^{-1}(X) u_0 =1+\sum_{n=1}^\infty b_n X^n$. For all $n\ge0$, we have $u_0^n b_n\in\OO_\LL$.\qed
\end{lemma}

\begin{noproof}{lemma}[{\cite[Lemma~2.3]{HondaCommFormalGroups}}] \label{Honda2.3}
    Let $f\in\LL[[X]]_0$ have relaxed Honda type $u\in\OO_\LL[[X]]$. Let $\psi(X)=\sum_{n=1}^\infty \psi_n X^n\in\LL[[X]]_0$ be a power series such that $\psi_0,\ldots,\psi_{r-1}\in\OO_\LL$ for some $r\ge2$. 
    For all $v\in\OO_\LL[[X]]$, we have
    \[v*(f\circ\psi)\equiv (v*f)\circ\psi \pmod{X^{r+1},\pi}. \qedhere\]
\end{noproof}

\begin{lemma}[{\cite[Lemma~2.4]{HondaCommFormalGroups}}] \label{Honda2.4}
    If $f\in \LL[[X]]_0$ and $g\in\LL[[X]]_0$ both have relaxed Honda type $u\in\OO_\LL[[X]]$, then $g^{-1}\circ f\in\OO_\LL[[X]]_0$. \qed
\end{lemma}

Using these lemmata, one can prove the following:

\begin{theorem}[{\cite[Theorem~2]{HondaCommFormalGroups}}] \label{HondaThm2}
    If $f\in \LL[[X]]_0$ and $g\in\LL[[X]]_0$ both have relaxed Honda type $u\in\OO_\LL[[X]]$, then their associated formal groups are strongly isomorphic over $\OO_\LL$, i.e., there is an isomorphism of formal groups $\varphi:f\to g$ defined over $\OO_\LL$ such that $\varphi\equiv X\mod X^2$. \qed
\end{theorem}

\subsection{Demchenko types}
Demchenko \cite{Demchenko} established a generalization of Honda theory that is equipped to handle certain kinds of tame ramification. We briefly recall the relevant notions from his work.

As before, let $\LL/\QQ_p$ be a finite extension with ramification index $e$, uniformizer $\pi\in\OO_\LL$, and residue field of order $q$. Let $\LL'$ be the maximal unramified extension of $\QQ_p$ inside $\LL$, and let $\rho\in\OO_{\LL'}$ be a uniformizer. 
The uniformizer $\pi$ defines a power integral basis of $\LL$ over $\LL'$, that is, every element $a\in\LL$ can be written uniquely as
\[a=\sum_{i=0}^{e-1} a_{[i]} \pi^i,\]
where $a_{[i]}\in\LL'$ for all $i=0,\ldots,e-1$.
This decomposition can be extended to power series as follows: for $f \in\LL[[X]]$, we write
\[f(X)=\sum_{n=0}^\infty a_nX^n=\sum_{n=0}^\infty \sum_{i=0}^{e-1} a_{n,[i]} \pi^i X^n = \sum_{i=0}^{e-1} \left(\sum_{n=0}^\infty a_{n,[i]} X^n\right) \pi^i \equalscolon \sum_{i=0}^{e-1} f_{[i]}(X) \pi^i,\]
where $f_{[i]}\in\OO_{\LL'}[[X]]$ for all $i=0,\ldots,e-1$.
We define a Frobenius action $\upphi$ on $\LL'[[X]]$ as follows. Let $\upphi$ act on $\LL'$ as the Frobenius of the residue fields $\FF_q/\FF_p$, and let $\upphi(X)\colonequals X^q$.
\begin{definition}
    A power series $f\in\LL[[X]]_0$ is said to be of \emph{Demchenko type} $(A_0,\ldots,A_{e-1})\in\OO_{\LL'}[[\upphi]]^e$ if $f(X)\equiv X\pmod{X^2\LL[[X]]}$ and $\pi f_{[i]}\equiv A_i* f_{[0]} \pmod{\pi\OO_\LL[[X]]}$.
\end{definition}

\begin{theorem}[{\cite[Theorem~3]{Demchenko}}] \label{DemchenkoThm3}
    If $e<p$ and $f\in\LL[[X]]_0$ is the logarithm of a one-dimensional formal group over $\OO_\LL$, then $f$ has a Demchenko type.
\end{theorem}

We have the following generalization of \cite[Lemma~4.1]{HondaCommFormalGroups}:
\begin{lemma} \label{Honda4.1}
    Suppose that $e<p$.
    Let $f\in\LL[[X]]_0$ be of Demchenko type $(A_0,\ldots,A_{e-1})$, and let $\psi\in\OO_\LL[[X]]_0$. Then $f^{-1}(\pi\psi(X))\equiv 0\pmod{\pi\OO_\LL[[X]]}$.
\end{lemma}
\begin{proof}
    The proof is analogous to that of Honda.
    Consider the series
    \begin{equation} \label{eq:h-Demchenko}
        h\colonequals \left(1-\frac{A_0}{\rho}\right)^{-1}*\id+\sum_{i=1}^{e-1} \pi^i \cdot\frac{A_i}{\rho} \cdot\left(1-\frac{A_0}{\rho}\right)^{-1}*\id.
    \end{equation}
    This is a series of Demchenko type $(A_0,\ldots,A_{e-1})$. By \cite[Lemma~2.1]{Demchenko}, we have $h^{-1}\circ f\in\OO_\LL[[X]]_0$ (the assumptions of the lemma are satisfied since $e<p$). Therefore, it suffices to show that $\ell(X)\colonequals h^{-1}(\pi X)\equiv 0\pmod{\pi\OO_\LL[[X]]}$.
    Let us write
    \[h(X)=\sum_{n=0}^\infty b_n X^{q^n}=\sum_{i=0}^{e-1}\sum_{n=0}^\infty b_{n,[i]} X^{q^n},\]
    with $b_n\in\LL$ and $b_{n,[i]}\in\LL'$ for $i=0,\ldots,e-1$. Since $\ell(X)\equiv \pi X\pmod{X^2}$, we find that the linear coefficient of $\ell$ is in $\pi\OO_\LL$. We proceed by induction: let $r\ge2$, and suppose the coefficients of $X$, \ldots, $X^{r-1}$ in $\ell$ are in $\pi\OO_\LL$. Then we can write $\ell(X)=\pi\ell^{(r)}(X)+\bar\ell_r(X)$ with $\ell^{(r)}\in\OO_\LL[[X]]$ and $\bar\ell_r\equiv 0\pmod{X^r \OO_\LL[[X]]}$.
    Taking the identity $h(\ell(X))=\pi X$ modulo $X^{r+1}$, we find
    \[\ell(X)+\sum_{n=1}^{r-1} \pi^{q^n} b_n \ell^{(r)}(X)^{q^n} \equiv \pi X\pmod{X^{r+1}\OO_\LL[[X]]}.\]
    Applying \cref{Honda2.2} to $u=(1+A_i/\rho)^{-1}$ for each $i=0,\ldots,e-1$, we have $\pi^{q^n} b_n\in\pi\OO_\LL$ for each $n\ge1$. Therefore $\ell(X)\equiv 0\pmod{(X^{r+1},\pi)}$. The claim follows by taking the limit as $r\to\infty$.
\end{proof}

This allows us to generalize \cite[Lemma~4.2]{HondaCommFormalGroups}.
\begin{lemma} \label{Honda4.2}
    Suppose that $e<p$.
    Let $f\in\LL[[X]]_0$ be of Demchenko type $(A_0,\ldots,A_{e-1})$, and let $\psi_1\in\LL[[X]]_0$ and $\psi_2\in\OO_\LL[[X]]_0$. Then $f\circ\psi_1\equiv f\circ\psi_2\pmod{\pi\OO_\LL[[X]]}$ if and only if $\psi_1\equiv\psi_2\pmod{\pi\OO_\LL[[X]]}$.
\end{lemma}
\begin{proof}
    First suppose that $\psi_1\equiv\psi_2\pmod{\pi\OO_\LL[[X]]}$. Then clearly $\psi_1\in\OO_\LL[[X]]_0$. Let $h$ be the power series defined in \eqref{eq:h-Demchenko}. This has the same Demchenko type as $f$, therefore \cite[Lemma~2.1]{Demchenko} states that $h^{-1}\circ f\in\OO_\LL[[X]]_0$, and thus $$(h^{-1}\circ f)\circ\psi_1\equiv (h^{-1}\circ f)\circ \psi_2 \pmod{\pi\OO_\LL[[X]]}.$$ Using \cite[Lemma~2.1]{HondaCommFormalGroups} and \cref{Honda2.2}, we find that 
    \[f\circ\psi_1=h\circ(h^{-1}\circ f)\circ\psi_1\equiv h\circ(h^{-1}\circ f)\circ \psi_2=f\circ\psi_2 \pmod{\pi\OO_\LL[[X]]},\] 
    as desired.

    Conversely, suppose that $f\circ\psi_1\equiv f\circ\psi_2\pmod{\pi\OO_\LL[[X]]}$. Define $g\colonequals\frac1\pi f^{-1}(f\circ\psi_1-f\circ\psi_2)\in\LL[[X]]_0$. By \cref{Honda4.1}, we actually have $g\in\OO_\LL[[X]]_0$. By the definition of $g$, we have $\psi_1=F(\psi_2,\pi g)$ where $F$ is the formal group law associated with $f$, defined as $F(x,y)=f^{-1}(f(x)+f(y))$. Since both $F$ and $g$ have coefficients in $\OO_\LL$, this shows that $\psi_1\equiv\psi_2\pmod{\pi\OO_\LL[[X]]}$, as desired.
\end{proof}

\section{Plus/minus universal norms}\label{sec:points}
The goal of this section is to construct an explicit height-two formal groups for a specific tamely ramified extension of $\Qp$. We will construct explicit points over the cyclotomic extension of this extension that satisfying the desired  ``jumping trace'' relations.

Let $e\ge2$ be an integer and $p\ge5$ be a prime number such that the following holds:
\begin{equation} \label{eq:e-conditions}
         e \mid (p+1).
\end{equation}
Let $\pi$ be an $e$-th root of $-p$, and let $\LL\colonequals\QQ_{p^2}(\pi)$. Note that $\Qpt$ contains all $e$-th roots of unity because of \eqref{eq:e-conditions}.  Let $\rho$ be a fixed uniformizer in $\QQ_{p^2}$. (In this section, we work with a general uniformizer $\rho$; in the application to elliptic curves, this will be simply $p$.) For $n\ge0$, define $\omega_n(X)=(1+X)^{p^n}-1$. Note that $\omega_n=\omega_1\circ\ldots\circ\omega_1$ is the $n$-fold iteration of $\omega_1$.

Kim and Park \cite[\S2]{KimPark} considered a certain formal group law over $\QQ_{p^2}$, and constructed plus/minus universal norms over the cyclotomic $\ZZ_p$-extension of $\QQ_{p^2}$. 
The system $(\zeta_{p^n}-1)_{n\ge1}$, where $\zeta_{p^n}$ is a primitive $p^n$-th root of unity and $(\zeta_{p^{n}})^p=\zeta_{p^{n-1}}$, is crucial to their construction: these are uniformizers for the fields occurring in the cyclotomic extensions $\QQ_{p^2}(\zeta_{p^n})$, and they are connected by the polynomials $\omega_1$ due to the formula 
\begin{equation} \label{eq:compatibility-zeta-Phi}
    \omega_1(\zeta_{p^n}-1)=\zeta_{p^{n-1}}-1 \text{ for } n\ge2.
\end{equation}
Kim and Park then consider the power series 
\[\ell_{\mathrm{KP}}(X)\colonequals \sum_{i=0}^\infty \frac{\omega_{2i}(X)}{\rho^i}  \in \QQ_{p^2}[[X]],\]
where $\rho$ is a uniformizer in $\Qpt$.
Since the numerators are iterates of $\omega_2$, this makes the power series $\ell_\mathrm{KP}$ well suited for constructing special points over $\QQ_{p^2}(\zeta_{p^n})$ satisfying a ``jumping'' condition with respect to the trace map.
Furthermore, Kim and Park showed that the formal group associated with $\ell_{\mathrm{KP}}$ is of height $2$, has Honda type $X^2-\rho$, and it is Lubin--Tate with parameter $\rho$.

As an aside, we note that several related but different constructions of universal norms exist in the Iwasawa theory of elliptic curves. The first such construction is due to Kobayashi \cite[\S8]{Kobayashi} concerning the cyclotomic $\Zp$-extension of $\Qp$. This was generalized to the cyclotomic $\Zp$-extension of an unramified extension of $\Qp$ by Kim \cite[\S3]{BDKimCompositio}, which subsequently was further extended by Kitajima and Otsuki \cite[\S3]{KitajimaOtsuki}. Kim also constructed universal norms for the $\Zp^2$-extension of the completion of a quadratic imaginary field at prime that is either split \cite[\S2]{BDKimZp2split} or inert \cite[\S2]{BDKimZp2inert}.

Below, we adapt the construction of Kim and Park to $\LL$. As the extension $\LL/\QQ_{p^2}$ is totally and tamely ramified, one first needs to construct a suitable system of uniformizers. This comes at a cost: the compatibility relation will be given by a less explicit power series $\Omega_2$ instead of the polynomial $\omega_2$. Moreover, the power series will have coefficients in $\LL$ as opposed to $\QQ_{p^2}$.

\subsection{A system of uniformizers}
Let us write $\KK\colonequals\QQ_{p^2}$. 
Define $\KK_n\colonequals\KK(\zeta_{p^n})$ and $\LL_n\colonequals\KK_n\LL=\KK_n(\pi)=\KK_n(\sqrt[e]{-p})$.

\begin{lemma}\label{lem:Kummer}
   Let $n\ge1$. Then we have $[\LL_n:\KK_n] = e $ if $e$ is odd. Otherwise, $[\LL_n:\KK_n]=e/2$.
\end{lemma}
\begin{proof}
The extension $\LL_n/\KK_n$ is totally ramified of degree at most $e$.
    If $e$ is odd, \eqref{eq:e-conditions} implies that $\varphi(p^n)$ and $e$ are coprime integers. In particular, $\KK_n$ does not contain a $d$-th root of $-p$ for any divisors $d$ of $e$. Thus, $[\LL_n:\KK_n]=e$.

    We now assume that $e$ is even. Recall that $\sqrt{(-1)^{\frac{p-1}2}p}\in\QQ(\zeta_p)$. In particular, if $p\equiv 3\mod 4$, we have $\sqrt{-p}\in\QQ_p(\zeta_p)$. If $p\equiv 1\mod 4$, then $\sqrt{-1}\in \QQ_p$ (by Hensel's lemma), so we equally have $\sqrt{-p}\in\QQ_p(\zeta_p)$. As $e\mid p+1$, the greatest common divisor of $\varphi(p^n)$ and $e$ is $2$. Hence, $\KK_n$ does not contain a $d$-th root of $-p$ for any divisors $d$ of $e$ except $d=2$. Thus, $[\LL_n:\KK_n]=e/2$.
\end{proof}

\begin{remark}
    Note that for $e=2$, we have $\LL\subseteq \KK_1$. In this case, the formal group $\F$ we are about to construct will agree with that of \cite[\S2]{KimPark}. For $e\ne2$, the fields $\KK_n$ and $\LL$ are disjoint over $\KK$.
\end{remark}

In what follows, we write 
\[
e'=\begin{cases}
    e& \text{if $e$ is odd},\\
    e/2& \text{if $e$ is even}.
\end{cases}
\]
For $n\ge1$, let $\pi_n=(1-\zeta_{p^n})^{1/e'}$.

\begin{lemma}
    We have $\LL_n=\KK_n(\pi_n)$. Furthermore, $\pi_n$ is a uniformizer in $\LL_n$.
\end{lemma}
\begin{proof}
We first consider the case where $e$ is odd. Lemma~\ref{lem:Kummer} tells us that 
$\LL_n/\KK_n$ is a degree $e$ Kummer extension with parameter $-p$. By Kummer theory, $\LL_n$ and $\KK_n(\pi_n)$ agree if and only if the parameters $-p$ and $1-\zeta_{p^n}$ generate the same subgroup in $\KK_n^\times/(\KK_n^\times)^e$. 
    
We have $$p=\prod_{a\in(\ZZ/p^n\ZZ)^\times}(1-\zeta_{p^n}^a).$$
Furthermore, $$\frac{1-\zeta_{p^n}^a}{1-\zeta_{p^n}}=1+\zeta_{p^n}+\cdots+\zeta_{p^n}^{a-1}\equiv a\mod (1-\zeta_{p^n}).$$ 
Thus, we deduce $$\frac{-p}{(1-\zeta_{p^n})^{\varphi(p^n)}}\equiv-\prod_{a\in(\ZZ/p^n\ZZ)^\times}a \equiv 1\mod (1-\zeta_{p^n})$$ by the generalized Wilson's theorem. Hence, $-p=u (1-\zeta_{p^n})^{\varphi(p^n)}$ for some $1$-unit $u\in U^1_{\KK_n}$. Since $p$ is coprime to $e$, the $e$-th power map is surjective on $1$-units, so $u\in\left( U_{\KK_n}^1\right)^e$.
The exponent $\varphi(p^n)$ is congruent to either $2$ or $-2$ modulo $e$, depending on the parity of $n$. 
Therefore, $-p$ and $(1-\zeta_{p^n})^2$ generate the same subgroup in $\KK_n^\times/(\KK_n^\times)^e$. In particular,
\[
\LL_n=\KK_n\left((1-\zeta_{p^n})^{2/e}\right)=\KK_n\left((1-\zeta_{p^n})^{1/e}\right),
\]
where the last equality follows from the fact that $e$ is odd. Since $1-\zeta_{p^n}$ is a uniformizer in $\KK_n$, it follows that $(1-\zeta_{p^n})^{1/e}$ is a uniformizer in $\LL_n$.

The case where $e$ is even can be proved similarly. Indeed, a similar argument shows that $\sqrt{-p}$ and $(1-\zeta_{p^n})$ generate the same subgroup in $\KK_n^\times/(\KK_n^\times)^{e'}$. Thus, we deduce from Kummer theory
\[
\LL_n=\KK_n\left((\sqrt{-p})^{1/e'}\right)=\KK_n\left((1-\zeta_{p^n})^{1/e'}\right)=\KK_n\left(\pi_n\right). \qedhere
\]
\end{proof}

The compatibility relation \eqref{eq:compatibility-zeta-Phi} can be written as
\begin{equation} \label{eq:compatibility-pi-Phi}
    \omega_1(-\pi_n^{e'})=-\pi_{n-1}^{e'} \text{ for }n\ge2.
\end{equation}
In order to obtain a compatibility relation between the uniformizers $\pi_n$ and $\pi_{n-1}$, we need to take $e'$-th roots. For this, recall that
\[(1+X)^{1/e'}=\sum_{j=0}^\infty \binom{\frac1{e'}}{j} X^j,\]
and note that
\[\omega_1(-X^{e'})=(1-X^{e'})^p-1=-pX^{e'}\cdot \left(1-\frac{\omega_1(-X^{e'})+pX^{e'}}{pX^{e'}}\right).\]
Therefore, the expression $\sqrt[e']{\omega_1(-X^{e'})}$ can be defined to mean
\begin{equation}\label{eq:Omega1}
    \Omega_1(X)\colonequals \pi X \sum_{j=0}^\infty \binom{\frac1{e'}}{j} \cdot \left(1-\frac{\omega_1(-X^{e'})+pX^{e'}}{pX^{e'}}\right)^j.
\end{equation}
Using \eqref{eq:compatibility-pi-Phi}, we deduce that the uniformizers $\pi_n$ can be chosen such that the following compatibility relation holds:
\begin{equation} \label{eq:compatibility-pi-Psi}
    \Omega_1(\pi_n)=\pi_{n-1} \text{ for }n\ge2.
\end{equation}
For $n\ge1$, we define $\Omega_n\colonequals\Omega_1\circ\ldots\circ\Omega_1$ as the $n$-fold iterate of $\Omega_1$. For $n=0$, let $\Omega_0(X)\colonequals X$. Since $\omega_1(\zeta_{p}-1)=0$, we have $\Omega_2(\pi_2)=\Omega_1(\pi_1)=0$. For notational coherence, we therefore define $\pi_0$ and $\pi_{-1}$ as zero (not to be confused with the uniformizer $\pi$ in $\LL$), and let $\LL_{-1}\colonequals\LL$.

\subsection{An explicit height-two formal group}
We define
\begin{equation} \label{eq:ell-def}
    \ell(X)\colonequals \sum_{i=0}^\infty \frac{\Omega_{2i}(X)}{\rho^i}\in\LL[[X]],
\end{equation}
and let $\F$ denote the formal group with logarithm $\ell$.

\begin{lemma} \label{ell-Honda-type}
    The power series $\ell$ has relaxed Honda type $X^2-\rho$.
\end{lemma}
\begin{proof}
    Recall that $\omega_1(X)\equiv X^p \pmod{\pi\OO_\LL[[X]]}$. Combining this with \eqref{eq:Omega1} shows that $\Omega_1(X)\equiv X^p \pmod{\pi\OO_\LL[[X]]}$.
    By iterating, we find that $\Omega_2(X)\equiv X^{p^2} \pmod{\pi\OO_\LL[[X]]}$, which implies $\ell(X^{p^2}) \equiv \ell(\Omega_2(X)) \pmod{\pi\OO_\LL[[X]]}$. Using the definition \eqref{eq:ell-def} of $\ell$, we also have $\ell(\Omega_2(X))\equiv \rho \ell(X) \pmod{\pi\OO_\LL[[X]]}$. Combining the last two congruences yields the claim.
\end{proof}

\begin{lemma} \label{ell-rho-p}
    The formal group $\F$ is of height $2$, and it is Lubin--Tate with parameter $\rho$. In particular, $\ell\left([\frac{\rho}{p}](X)\right)=\frac{\rho}{p}\ell(X)$.
\end{lemma}
\begin{proof}
    We may replicate the proof of \cite[Proposition~8.6]{Kobayashi}. We need to show that $[\rho]X\equiv \rho X\pmod{X^2\OO_\LL[[X]]}$ and $[\rho]X\equiv X^{p^2}\pmod{\pi\OO_\LL[[X]]}$. The first assertion is clear, and the second follows from $\ell^{-1}(\rho\ell(X))\equiv \ell^{-1}(\ell(X^{p^2}))\equiv X^{p^2} \pmod{\pi\OO_\LL[[X]]}$, which we have seen in the proof of \cref{ell-Honda-type}.
\end{proof}

\begin{lemma} \label{p-torsion-free}
   Suppose that $e'\ne p+1$. The abelian group $\F(\LL_\infty)$ is $p$-torsion free.
\end{lemma}

\begin{proof}
    The field $\LL_\infty$ is the cyclotomic $\ZZ_p$-extension of $\LL_1=\LL(\mu_p)$. The ramification index of $\LL_1$ over $\QQ_p$ is $e(\LL_1/\QQ_p)=(p-1)\cdot e'$ by Lemma~\ref{lem:Kummer}. Therefore, it follows from $e'\ne (p+1)$ and \eqref{eq:e-conditions} that $(p^2-1)$ does not divide the ramification index of $\LL_1/\QQ_p$. Since $\F$ is a Lubin--Tate formal group of height 2, any extension of $\QQ_p$ that contains a $p$-torsion of $\F$ has ramification index at least $p^2-1$ (see the proofs of \cite[Lemma~3.1]{FM} and \cite[Proposition~8.7]{Kobayashi}). Hence, the lemma follows.
\end{proof}

Note that the condition $e'\ne p+1$ always holds when $e$ is even by virtue of \eqref{eq:e-conditions}.

\begin{lemma} \label{ell-iso-max-ideal}
    Suppose that $e\ne p+1$. Then for all $j\ge1$, we have an isomorphism $\ell: \F(\mm_\LL^j)\xrightarrow{\sim}\mm_\LL^j$.
\end{lemma}
\begin{proof}
    Recall from \cite[Chapter~IV, Theorem~6.4]{Silverman2ed} that the asserted isomorphism holds whenever 
    \[j>\frac{v_\LL(p)}{p-1}=\frac{e}{p-1}.\]
    By \eqref{eq:e-conditions} and the assumption $e\ne p+1$, we have $e<p-1$. Therefore, the isomorphism holds for all $j\ge1$.
\end{proof}

For the rest of this section, we assume that
\begin{equation} \label{eq:e-ne-p+1}
    e\ne p+1.
\end{equation}
\begin{remark}
Note that when we apply these results to formal groups of elliptic curves later on, we will have $e\in\{2,3,4,6\}$, meaning that \eqref{eq:e-ne-p+1} will exclude only the case where $p=5$ and $e=6$.
\label{rk:ep+1}    
\end{remark}
By the definition of $e'$ and \eqref{eq:e-conditions}, $e\ne p+1$ implies $e'\ne p+1$. In particular, both \cref{p-torsion-free,ell-iso-max-ideal} apply under \eqref{eq:e-ne-p+1}.

\subsection{Sequences of points}
We shall inductively construct a sequence of points $b_n\in\F(\LL_n)$ for $n\ge0$ such that the even/odd index subsequences satisfy a trace condition. The points $b_n$ shall be of the form
\begin{equation} \label{eq:bn-def}
    b_n\colonequals \ell^{-1}(\varpi_n) \mathop{+}_\F \pi_n = \ell^{-1}(\varpi_n+\ell(\pi_n)),
\end{equation}
where $+_\F$ denotes the formal group law, and $\varpi_n\in\mm_{\LL_n}$ shall be chosen appropriately. Note that $\ell^{-1}(\varpi_n)$ is going to be well-defined by \cref{ell-iso-max-ideal}.

We begin with the two classes at the bottom of the tower. 
\begin{lemma} \label{b-generators}
    There are $\varpi_0\in\OO_\LL$ and $\varpi_1\in\OO_{\LL_1}$ such that $b_0$ and $\Tr^\F_{\LL_1/\LL} b_1$ both generate $\F(\LL)$ as an $\OO_\KK[\Gal(\LL/\KK)]$-module (where the addition is given by $\mathop{+}_\F$). 
\end{lemma}
\begin{proof}
    Since $\pi_0=0$, we have $\ell(b_0)=\varpi_0$. Thanks to \cref{ell-iso-max-ideal}, it is enough to show that $\varpi_0$ generates $\mm_\LL$ as an $\OO_\KK[\Gal(\LL/\KK)]$-module. This holds if $\varpi_0$ is any element of $\LL$ of valuation $1$ as $\LL/\KK$ is tamely ramified \cite{Kawamoto}. 
    
    For the assertion on $\varpi_1$, note that as power series in $X$, both $\ell$ and $\ell^{-1}$ have first-order approximation $X$. Thus,
    \begin{align*}
        \ell(\Tr_{\LL_1/\LL} b_1) &= \ell\left(\Tr_{\LL_1/\LL} \ell^{-1}\left(\varpi_1 + \ell(\pi_1)\right)\right) \\
        &\equiv \Tr_{\LL_1/\LL} \varpi_1 + \Tr_{\LL_1/\LL} \pi_1 \pmod{\pi^n},
    \end{align*}
    where $n=\lfloor\min\{ 2v_\LL(\varpi_1), v_\LL(\varpi_1)+v_\LL(\pi_1) , 2v_\LL(\pi_1)\}\rfloor$. As the extension $ \LL_1/\LL$ is tamely ramified, the trace map  $\Tr_{\LL_1/\LL}:\OO_{\LL_1}\to \OO_{\LL}$ is surjective. The assertion follows.
\end{proof}
\begin{remark}
    There is a result analogous to \cref{b-generators} in \cite[p.~216]{KimPark}: on level zero, this requires the additional assumption that $p\not\equiv\rho\pmod{p^2}$. No such condition is present here: as opposed to the construction in \textit{op.~cit.}, in which the analogue of the first addend in \eqref{eq:bn-def} is the same on all levels, we choose a different $\varpi_n$ for each $n$ implicitly.
\end{remark}
\begin{remark}
    In the case $e=2$, the results preceding \cref{b-generators} are exactly the same as in \cite[\S2]{KimPark}. Our definition of the points $b_n$ marks a departure from their construction: in our application, the elliptic curve will not have good reduction over $\KK$, only over $\LL$.
\end{remark}

We now proceed with the inductive construction. We will need the following:
\begin{lemma} \label{ell-pi-n-sigma}
    Let $n\ge2$ and $\sigma\in\Gal(\LL_{n}/\LL_{n-2})$. Then $\Omega_2(\pi_n^\sigma)=\pi_{n-2}$, and 
    \[\ell(\pi_n^\sigma)=\pi_n^\sigma+\frac1\rho\cdot \ell(\pi_{n-2}).\]
    In particular, we have
    \[\ell\left(\Tr^\F_{\LL_n/\LL_{n-2}} \pi_n\right)=\Tr_{\LL_n/\LL_{n-2}} (\pi_n)+\frac{[\LL_n:\LL_{n-2}]}{\rho} \cdot \ell(\pi_{n-2}),\]
    where $\Tr^\F_{\LL_n/\LL_{n-2}}$ is the trace map with respect to the formal group law $\F$.
\end{lemma}
\begin{proof}
    Under the isomorphism $\Gal(\LL_n/\LL_{n-2})\simeq \Gal(\KK_n/\KK_{n-2})$, the automorphism $\sigma$ corresponds to one given by $\zeta_{p^n}\mapsto\zeta_{p^n}\zeta_{p^2}^k$ for some integer $0\le k< p^2$. Since $\pi_n$ is an $e'$-th root of $1-\zeta_{p^n}$, the conjugate $\pi_n^\sigma$ is an $e'$-th root of $1-\zeta_{p^n}\zeta_{p^2}^k$. Therefore
    \[
    \Omega_1(\pi_n^\sigma)
    = \sqrt[e']{-\omega_1\!\left((-\pi_n^\sigma)^{e'}\right)}
    = \sqrt[e']{-\omega_1\!\left(\zeta_{p^n}\zeta_{p^2}^k - 1\right)}
    = \sqrt[e']{-\left(\left(\zeta_{p^n}\zeta_{p^2}^k\right)^p - 1\right)}
    = \sqrt[e']{1 - \zeta_{p^{n-1}}}
    = \pi_{n-1}.
    \]
    Since $\Omega_2=\Omega_1\circ\Omega_1$, using \eqref{eq:compatibility-pi-Psi}, we deduce that $\Omega_2(\pi_n^\sigma)=\pi_{n-1}$. The second assertion follows from this and the definition of $\ell$:
    \[\ell(\pi_n^\sigma)=\pi_n^\sigma+\sum_{i=1}^\infty \frac{\Omega_{2i}(\pi_n^\sigma)}{\rho^i}=\pi_n^\sigma+\sum_{i=1}^\infty \frac{\Omega_{2i-2}(\pi_{n-2})}{\rho^i}=\pi_n^\sigma +\frac1\rho\cdot \ell(\pi_{n-2}). \]
    In the last assertion, the left hand side is, by definition, equal to $\sum_{\sigma\in\Gal(\LL_n/\LL_{n-2})} \ell(\pi_n^\sigma)$, and now the claim follows from the second assertion.
\end{proof}

\begin{proposition} \label{Tr-ell-bn}
    For $n\ge2$, there exist $\varpi_n\in\mm_{\LL_n}$ such that 
    \[b_n\colonequals \ell^{-1}(\varpi_n)\mathop{+}_\F \pi_n\in\F(\LL_n)\]
    satisfies
    \[\Tr_{\LL_n/\LL_{n-1}}(\ell(b_n))=\frac{p}{\rho}\cdot\ell(b_{n-2}).\]
\end{proposition}
\begin{proof}
    We begin with the case $n=2$. The extension $\LL_2/\LL_1$ is of degree $p$, whereas the extension $\LL_1/\LL$ is of degree $p-1$. By the transitivity of the trace in towers of field extensions, it suffices for $\varpi_2$ to satisfy the following condition:
    \[\Tr_{\LL_2/\LL}(\ell(b_2))=\frac{p(p-1)}{\rho}\cdot\ell(b_{0}).\]
    Using \cref{ell-pi-n-sigma}, the left hand side is
    \[\Tr_{\LL_2/\LL}(\varpi_2+\ell(\pi_2))=\Tr_{\LL_2/\LL}(\varpi_2)+\ell\left(\Tr^\F_{\LL_2/\LL} \pi_2\right)=\Tr_{\LL_2/\LL}(\varpi_2)+\Tr_{\LL_2/\LL}(\pi_2)+\frac{p(p-1)}{\rho}\ell(\pi_0).\]
    Since $b_0=\varpi_0$ and $\pi_0=0$, this yields the condition
    \[\Tr_{\LL_2/\LL}(\varpi_2)+\Tr_{\LL_2/\LL}(\pi_2)=\frac{p(p-1)\varpi_0}{\rho}\]
    The trace map is additive, and $\Tr_{\LL_2/\LL}(\pi_2)$ is clearly a trace. Hence, the existence of $\varpi_2$ depends on whether the right hand side is a trace. According to \cite[V\S3, Lemma 4]{SerreLocalFields}, we have $\Tr_{\LL_2/\LL_1}(\mm_{\LL_2})=\mm_{\LL_1}^r$, where $r=\lfloor ((t+1)(p-1)+1)/p\rfloor$, and $t$ is the last ramification jump in the lower ramification group filtration of $\Gal(\LL_2/\LL_1)$. Via base change from \cite[IV\S4, Proposition 18]{SerreLocalFields}, we have $t=p-1$, and so $r=p-1$. In the tamely ramified extension $\LL_1/\LL_0$, the trace map is surjective, and we obtain $\Tr_{\LL_2/\LL}(\mm_{\LL_2})=\mm_\LL$. In particular, there is some $\varpi_2\in\mm_{\LL_2}$ satisfying the condition above.

    Let $n\ge3$, and suppose that $\varpi_{n-2}$ has already been defined. Now $\LL_n/\LL_{n-1}$ and $\LL_{n-1}/\LL_{n-2}$ are both cyclic extensions of degree $p$, and therefore it suffices to find $\varpi_n$ satisfying the following double-step trace condition:
    \[\Tr_{\LL_n/\LL_{n-2}}(\ell(b_2))=\frac{p^2}{\rho}\cdot\ell(b_{n-2}).\]
    Using \cref{ell-pi-n-sigma}, this is equivalent to 
    \[\Tr_{\LL_n/\LL_{n-2}}(\varpi_n)+\Tr_{\LL_n/\LL_{n-2}}(\pi_n)+\frac{p^2}{\rho}\ell(\pi_{n-2}) =\frac{p^2}{\rho}\varpi_{n-2}+\frac{p^2}{\rho}\ell(\pi_{n-2}).\]
    It follows that $\varpi_n$ exists if $p^2 \varpi_{n-2}/\rho\in\Tr_{\LL_n/\LL_{n-2}}\mm_{\LL_n}$. Observe that 
    \[v_{\LL_{n-2}} \left(\frac{p^2}{\rho}\varpi_{n-2}\right)\ge v_{\LL_{n-2}} \frac{p^2}{\rho}=v_{\LL_{n-2}}(p)=[\LL_{n-2}:\LL_1]\cdot [\LL_{1}:\LL] \cdot [\LL:\KK]=p^{n-3}\cdot (p-1)\cdot e.\]
    Applying \cite[IV\S4, Proposition 18]{SerreLocalFields} twice and noting that the last ramification jumps are at $p^{n-1}$ and $p^{n-2}$, we have $\Tr_{\LL_n/\LL_{n-2}}\mm_{\LL_n}=\mm_{\LL_{n-2}}^{r}$ with $r=2{p^{n-3}}(p-1)$. In particular, since $e\ge2$, we find that $p^2\varpi_{n-2}/\rho$ is a trace, as required. 
\end{proof}

\begin{corollary} \label{Tr-bn} For $n\ge3$, we have
    \[\Tr_{\LL_n/\LL_{n-1}}\left(\left[\frac{\rho}{p}\right](b_n)\right)=b_{n-2}.\]
\end{corollary}
\begin{proof}
    This follows from \cref{ell-rho-p,p-torsion-free,Tr-ell-bn}.
\end{proof}

Let $\LL'_\infty$ be the cyclotomic $\ZZ_p$-extension of $\LL$, that is, the unique $\ZZ_p$-extension of $\LL$ inside $\LL_\infty$. Let $\LL'_n$ be the $n$th layer of this extension, so that $\LL'_n$ is the unique $\ZZ/p^n\ZZ$-extension of $\LL$ inside $\LL_{n+1}$.
For $n\ge-1$, define
\[d_n\colonequals \Tr_{\LL_\infty/\LL_\infty'} \left(\left[\frac{\rho}{p}\right]^{\left\lfloor \frac{n+1}{2}\right\rfloor}(b_{n+1})\right),\]
and for $n\ge0$, let
\[d_n^+\colonequals\begin{cases}
    d_n & \text{ if $n$ is even}, \\
    d_{n-1} & \text{ if $n$ is odd}, \\
\end{cases} \quad \text{and} \quad d_n^-\colonequals\begin{cases}
    d_{n-1} & \text{ if $n$ is even}, \\
    d_{n} & \text{ if $n$ is odd}. \\
\end{cases}\]
\begin{corollary} 
\label{cor:generation}The sequence $d_n^\pm$ has the following properties:
    \begin{enumerate}[label=(\roman*)]
        \item Each of the points $d_{-1}^-$ and $d_0^+$ generates $\F(\LL)$ as an $\OO_\KK[\Gal(\LL/\KK)]$-module.
        \item For each $n\ge2$, there are trace compatibility relations
            \[\Tr_{\LL_n'/\LL_{n-1}'} d_n^\pm = d_{n-2}^\pm\]
        where $(-1)^{n}=\pm1$.
    \end{enumerate}
\end{corollary}
\begin{proof}
    Both statements follow from the definition of $d_n^\pm$ and the results established above: the former comes from \cref{ell-iso-max-ideal,b-generators}, whereas the latter comes from \cref{Tr-bn}.
\end{proof}

For any character $\chi$ of $\Gal(\LL'_n/\LL)$, define 
\[\tau_n(\chi)\colonequals \sum_{\sigma\in \Gal(\LL_n/\LL)} \chi(\sigma) \sigma(\varpi_n+\pi_n).\] 
We have the following adaptation of \cite[Proposition~2.9]{KimPark}:
\begin{lemma} \label{KP2.9}
    For a primitive character $\chi$ of $\Gal(\LL'_n/\LL)$, we have
    \[\sum_{\sigma\in \Gal(\LL'_n/\LL)} \chi(\sigma) \ell(d_n^\sigma)=\left(\frac \rho p\right)^{\left\lfloor\frac{n+1}{2}\right\rfloor} \tau_{n+1}(\chi).\]
\end{lemma}
\begin{proof}
    We first compute the analogous sum for $\left[{\rho}/{p}\right]^{\left\lfloor {(n+1)}/{2}\right\rfloor}(b_{n+1})$. Let $\widetilde\chi$ denote the inflation of $\chi$ from $\Gal(\LL'_n/\LL)$ to $\Gal(\LL_n/\LL)$.
    \begin{align*}
        &\phantom=\sum_{\sigma\in \Gal(\LL_n/\LL)} \widetilde\chi(\sigma) \ell\left(\sigma\left(\left[\frac{\rho}{p}\right]^{\left\lfloor \frac{n+1}{2}\right\rfloor}(b_{n+1})\right)\right) 
        = \left(\frac{\rho}{p}\right)^{\left\lfloor \frac{n+1}{2}\right\rfloor} \sum_{\sigma\in \Gal(\LL_n/\LL)} \widetilde\chi(\sigma) \ell\left(b_{n+1}^\sigma\right) \\
        &= \left(\frac{\rho}{p}\right)^{\left\lfloor \frac{n+1}{2}\right\rfloor} \sum_{\sigma\in \Gal(\LL_n/\LL)} \widetilde\chi(\sigma) \Big(\varpi_{n+1}^\sigma + \ell(\pi_{n+1}^\sigma)\Big) \\
        &= \left(\frac{\rho}{p}\right)^{\left\lfloor \frac{n+1}{2}\right\rfloor} \sum_{\sigma\in \Gal(\LL_n/\LL)} \widetilde\chi(\sigma) \Big(\varpi_{n+1}^\sigma + \pi_{n+1}^\sigma\Big) + \left(\frac{\rho}{p}\right)^{\left\lfloor \frac{n+1}{2}\right\rfloor} \sum_{i=1}^\infty \frac{1}{\rho^i}\sum_{\sigma\in\Gal(\LL_n/\LL)}\widetilde\chi(\sigma) \Omega_{2i}(\pi^\sigma_{n+1}).
    \end{align*}
    The first step is \cref{ell-rho-p}, followed by the definition of $b_{n+1}$.
    The character $\chi$ is primitive, but in the inner summation in the last term, the action on $\pi_{n+1}$ factors through $\Gal(\LL_{n-2}/\LL)$ because of \eqref{eq:compatibility-pi-Psi}. Hence, this inner sum vanishes by the orthogonality relations.
    The claim follows by taking traces as in the definition of $d_n$.
\end{proof}

\subsection{Plus/minus subgroups}
Consider the following plus and minus Selmer groups defined by the ``jumping trace'' conditions.

\begin{definition} \label{def:F-plus-minus}
For an integer $n\ge0$, we define
    \begin{align*}
        \F^+(\LL'_n) &\colonequals \left\{ P\in \F(\LL'_n) : \Tr_{\LL'_n / \LL'_{m+1}} P \in \F(\LL'_m) \text{ for all }  0\le m<n,\,2\mid m\right\}, \\
        \F^-(\LL'_n) &\colonequals \left\{ P\in \F(\LL'_n) : \Tr_{\LL'_n / \LL'_{m+1}} P \in \F(\LL'_m) \text{ for all }  -1\le m<n,\, 2\nmid m\right\}.
    \end{align*}
    Furthermore, define
\[
\F^\pm(\LL'_\infty)=\bigcup_{n\ge0}\F^\pm(\LL_n').
\]
\end{definition}

Let $G_n\colonequals\Gal(\LL'_n/\LL)$. Let $\Gamma\colonequals\Gal(\LL'_\infty/\LL)\simeq\ZZ_p$ be the Galois group of the cyclotomic $\ZZ_p$-extension $\LL'_\infty/\LL$. 
If $\OO$ is the ring of integers in some finite extension of $\QQ_p$, define $\Lambda^\OO\colonequals\Lambda^\OO(\Gamma)\colonequals\OO\llbracket\Gamma\rrbracket$. When $\OO=\Zp$, we simply write $\Lambda$. We fix a topological generator $\gamma$ of $\Gamma$. Set $X=\gamma-1$, which allows us to identify $\OO\llbracket\Gamma\rrbracket$ with the set of power series $\OO\llbracket X\rrbracket$.

For $i\ge0$, let $\Phi_i$ denote the $p^i$th cyclotomic polynomial.
The polynomial $\omega_n(X)=(1+X)^{p^n}-1$ is the product of the cyclotomic polynomials $\Phi_i$ for $0\le i\le n$. We define
\[\omega_n^\pm(X) \colonequals X \prod_{\substack{1\le i\le n \\ (-1)^i=\pm1}} \Phi_i, \quad \omegat_n(X)\colonequals\frac{\omega_n(X)}{X}.\]
Let $\Lambda^\OO_n\colonequals\Lambda^\OO/\omega_n\Lambda^\OO\simeq\OO[G_n]$ and $\Lambda^{\OO,\pm}_n\colonequals\Lambda^\OO/\omega^\pm_n\Lambda^\OO\simeq \omegat^\mp_n\Lambda^\OO_n$.
\subsection{Coleman maps} \label{sec:Coleman-maps}

Let $\T=\varprojlim \F[p^n]$ be the Tate module of $\F$. By Lubin--Tate theory, $\T$ is a free $\OO_\KK$-module of rank one, equipped with a Lubin--Tate character $\chi_\F: G_\KK\to \OO_\KK^\times$. Let $\tilde\chi_\F=\chi_\F^{-1}\chi_\cyc$, where $\chi_\cyc$ is the cyclotomic character. 
We define
\[
\tilde\T=\OO_\KK(\tilde\chi_\F).
\]

There is a perfect pairing of $G_\KK$-modules
\[
\tilde\T\times\F[p^\infty]\longrightarrow \KK/\OO_\KK(1), 
\]
which gives the local Tate pairing
\begin{equation}
    H^1(M, \tilde\T) \times H^1(M, \F[p^\infty]) \longrightarrow H^2(M,\KK/\OO_K(1))\xlongrightarrow{\sim}\KK/\OO_\KK
    \label{eq:KP-pairing}
\end{equation}
for any finite extension $M/\KK$.
We define $H^1_\pm(\LL'_n,\tilde\T)$ as the orthogonal complement of $\F^\pm(\LL'_n)\otimes\QQ_p/\ZZ_p$ under this pairing. 

Let
$$\langle-,-\rangle_n: H^1(\LL'_n,\T)\times H^1(\LL'_n,\tilde\T)\longrightarrow H^2(\LL'_n,\OO_\KK(1))\xlongrightarrow{\sim}\OO_\KK$$ 
be the pairing given by the cup product, similar to the one given in \cite[after Proposition~3.1]{KimPark}.

Let $P^\pm_n$ be the Perrin-Riou map associated with $d^\pm_n$, that is,
\[P^\pm_n: H^1(\LL'_n,\tilde \T)\longrightarrow \Gal(\LL'_n/\KK), \quad P^\pm_n(z)=\sum_{\sigma\in \Gal(\LL'_n/\KK)} \langle \sigma(d^\pm_n), z\rangle_n\cdot \sigma,\]
where $\F(\LL'_n)\hookrightarrow H^1(\LL'_n,\T)$ under the Kummer map.
\begin{proposition}\label{prop-coln}
    For each $n\ge0$, there exists a unique homomorphism \[\Col^\pm_n: H^1(\LL'_n,\tilde\T)\longrightarrow \Lambda^{\OO_\KK,\pm}_n[\Gal(\LL/\KK)]\] such that the following diagram commutes:
    \[\begin{tikzcd}
    	{H^1(\LL_n',\tilde\T)} & {\Lambda^{\OO_\KK,\pm}_n[\Gal(\LL/\KK)]} \\
    	{\displaystyle\frac{H^1(\LL'_n,\tilde\T)}{H^1_\pm(\LL'_n,\tilde\T)}} & {\Lambda^{\OO_\KK}_n[\Gal(\LL/\KK)]}
    	\arrow["{\Col^\pm_n}", dashed, from=1-1, to=1-2]
    	\arrow[two heads, from=1-1, to=2-1]
    	\arrow["{\omegat^\mp_n}", hook, from=1-2, to=2-2]
    	\arrow["{P^\pm_n}", from=2-1, to=2-2]
    \end{tikzcd}\]
    Moreover, each $\Col^\pm_n$ is surjective.
\end{proposition}
\begin{proof}
    The proof of \cite[Proposition~2.12]{KimPark} applies. Surjectivity comes from the properties of the bottom classes $d^-_{-1}$ and $d^+_0$ established in Corollary~\ref{cor:generation} and Nakayama's lemma.
\end{proof}

Define the Iwasawa cohomology as $H^1_\Iw(\LL'_\infty,\tilde\T)\colonequals\varprojlim_n H^1(\LL'_n,\tilde\T)$. The plus/minus Coleman maps are compatible with respect to the corestriction and projection maps because the Perrin-Riou maps are, see \cite[Proposition~8.21]{Kobayashi}. This defines plus/minus Coleman maps
\[\Col_\LL^\pm: H^1_\Iw(\LL'_\infty,\tilde\T)\twoheadrightarrow \varprojlim_n \Lambda^{\OO_\KK,\pm}_n[\Gal(\LL/\KK)]=\Lambda^{\OO_\KK}[\Gal(\LL/\KK)].\]

Let $\KK'_\infty$ be the cyclotomic $\Zp$-extension of $\KK$ and denote by $\KK'_n/\KK$ the sub-extension of degree $p^n$. Define $\F^\pm(\KK_n')$ and $\F^\pm(\KK_\infty')$ in the same way as $\F^\pm(\LL_n')$ and $\F^\pm(\LL_\infty')$ given in \cref{def:F-plus-minus}.

Let $H^1_\Iw(\KK'_\infty,\tilde\T)\colonequals\varprojlim_n H^1(\KK'_n,\tilde\T)$. We have similarly the surjective Coleman maps
\[\Col_\KK^\pm: H^1_\Iw(\KK'_\infty,\tilde\T)\twoheadrightarrow \varprojlim_n \Lambda^{\OO_\KK,\pm}_n=\Lambda^{\OO_\KK}.\]

\subsection{The structure of plus/minus subgroups}
\begin{proposition} \label{dpmn-generator}
    For all $n\ge0$, the following hold:
    \begin{align}
        \Hom\left(\F^\pm(\LL_n')\otimes \QQ_p/\ZZ_p, \QQ_p/\ZZ_p\right) &\simeq \omegat_n^\mp \OO_\KK[\Gal(\LL_n'/\KK)]; \label{eq:dpmn-1} \\
        \F^\pm(\LL'_n)\otimes \QQ_p/\ZZ_p &= \OO_\KK[\Gal(\LL_n'/\KK)](d_n^\pm)\otimes \QQ_p/\ZZ_p. \label{eq:dpmn-2}
    \end{align}
\end{proposition}
\begin{proof}
    These statements are byproducts of the construction of Coleman maps, and the same strategy as in \cite[Proposition~2.12]{KimPark} works. Let 
    \[D_n^\pm\colonequals\langle \sigma(d_n^\pm) : \sigma\in G_n\rangle_{\QQ_p}\le \F^\pm(\LL'_n)\]
    be the subgroup of $\F^\pm(\LL'_n)$ consisting of $\QQ_p$-linear combinations of $\sigma(d_n^\pm)$ with $\sigma\in G_n$.
    Surjectivity of the plus/minus Coleman maps implies that there is an isomorphism 
    \begin{equation} \label{eq:KP2.12-auxiliary}
        \Hom\left(D^\pm_n\otimes \QQ_p/\ZZ_p, \QQ_p/\ZZ_p\right) \simeq \omegat_n^\mp \OO_\KK[\Gal(\LL_n'/\KK)].
    \end{equation} 
    Now \eqref{eq:KP2.12-auxiliary} implies \eqref{eq:dpmn-2}. Finally, combining \eqref{eq:dpmn-2} with \eqref{eq:KP2.12-auxiliary}, one can observe that $D_n^\pm\otimes\QQ_p/\ZZ_p=\F^\pm(\LL_n')\otimes\QQ_p/\ZZ_p$, which shows that \eqref{eq:KP2.12-auxiliary} is equivalent to \eqref{eq:dpmn-1}.
\end{proof}

The following is an analogue of \cite[Theorem~7.1]{PollackRubin} and \cite[Proposition~3.3]{KimPark}. For a character $\chi\in\widehat G_n$ and a polynomial $g(X)\in\ZZ_p[[X]]$, we define $\chi(g)$ as follows. Extend $\chi$ to $\ZZ_p[G_n]$ by $\ZZ_p$-linearity, and evaluate it on the image of $g$ under the map $\ZZ_p[[X]]\simeq \Lambda(\Gamma)\twoheadrightarrow \ZZ_p[G_n]$, where the second map is induced by the canonical projection map.
\begin{corollary} \label{cor:cofree}
    The $\Lambda^{\OO_\KK}[\Gal(\LL/\KK)]$-module $\Hom\left(\F^\pm(\LL_\infty')\otimes \QQ_p/\ZZ_p, \QQ_p/\ZZ_p\right)$ is free of rank one, with a generator $f^\pm$ satisfying
    \[ \sum_{\sigma\in G_n} \chi(\sigma) \cdot f^\pm\left(\sigma(d_n^\pm) \otimes p^{-k} \right) = \chi(\omegat_n^\mp) \cdot p^{-k}\]
    for all characters $\chi\in\widehat G_n$ and integers $k,n\ge0$.
\end{corollary}
\begin{proof}
    Let $f^\pm_n\in \Hom\left(\F^\pm(\LL_n')\otimes \QQ_p/\ZZ_p, \QQ_p/\ZZ_p\right)$ be the element corresponding to $\omegat_n^\pm$ under the isomorphism of \cref{dpmn-generator}, and let $f^\pm\colonequals\varprojlim_n f_n^\pm$. Since the transition maps on both sides of the isomorphism are compatible, the claim follows.
\end{proof}

\begin{lemma} \label{Kob8.12}
    For every $n\ge0$, there are short exact sequences
    \begin{align*}
        0\longrightarrow \F(\LL)\otimes\QQ_p/\ZZ_p\xrightarrow{\mathrm{diag}} \left(\F^+(\LL'_n)\otimes\QQ_p/\ZZ_p\right) \oplus \left(\F^-(\LL'_n)\otimes\QQ_p/\ZZ_p\right) \longrightarrow \F(\LL'_n)\otimes\QQ_p/\ZZ_p\longrightarrow0 \\
        \intertext{and}
        0\longrightarrow \F(\KK)\otimes\QQ_p/\ZZ_p\xrightarrow{\mathrm{diag}} \left(\F^+(\KK_n)\otimes\QQ_p/\ZZ_p\right) \oplus \left(\F^-(\KK_n)\otimes\QQ_p/\ZZ_p\right) \longrightarrow \F(\KK_n)\otimes\QQ_p/\ZZ_p\longrightarrow0,
    \end{align*}
    where the leftmost nonzero arrows denote the diagonal embeddings.
\end{lemma}
\begin{proof}
    We begin by proving the first short exact sequence.
    This is equivalent to the validity of the following two assertions:
    \begin{align}
        \left(\F^+(\LL'_n)\otimes\QQ_p/\ZZ_p\right) \cap \left(\F^-(\LL'_n)\otimes\QQ_p/\ZZ_p\right) & = \F(\LL)\otimes\QQ_p/\ZZ_p, \label{eq:Kob8.12-intersection} \\
        \left(\F^+(\LL'_n)\otimes\QQ_p/\ZZ_p\right) + \left(\F^-(\LL'_n)\otimes\QQ_p/\ZZ_p\right) & = \F(\LL'_n)\otimes\QQ_p/\ZZ_p. \label{eq:Kob8.12-sum}
    \end{align}
    The first assertion \eqref{eq:Kob8.12-intersection} is, in fact, true even before tensoring with $\QQ_p/\ZZ_p$. Indeed, we may repeat the first half of the proof of \cite[Proposition~8.12(ii)]{Kobayashi}. This shows that if $P\in \F^+(\LL'_n)\cap\F^-(\LL'_n)$ and $P\in\F(\LL'_m)$ for some $0\le m\le n$, then the definition of $\F^\pm$ shows that for all $\sigma\in\Gal(\LL'_m/\LL'_{m-1})$, we have
    \[p^{n-m-\frac{1+(-1)^{n-m}}{2}}(\sigma(P)-P)=0.\]
    By $p$-torsion freeness (\cref{p-torsion-free}), it follows that $P\in\F(\LL'_{m-1})$, and by induction on $m$, we obtain $P\in\F(\LL'_{-1})=\F(\LL)$, as desired.
    Now the second assertion \eqref{eq:Kob8.12-sum} follows from \cref{dpmn-generator} by comparing $\OO_\LL$-coranks. This finishes the proof of the first short exact sequence.

    The proof of the second exact sequence is similar. The validity of \eqref{eq:Kob8.12-intersection} with $\LL'_n$ replaced with $\KK_n$ once again follows by Kobayashi's argument. For the analogue of \eqref{eq:Kob8.12-sum}, we use the observation that it follows from \cref{dpmn-generator} that $\F^\pm(\KK_n)\otimes\Qp/\Zp$ has $\OO_\KK$-corank one.
\end{proof}

\begin{remark}
 \label{rk:base-change}
    By the same proof as \cite[Proposition~8.18]{Kobayashi}, the kernel of the map $P_n^\pm$ in \cref{prop-coln} is given by $H^1_\pm(\LL_n',\T)$. Thus, $\ker(\Col_\LL^\pm)=\varprojlim H^1_\pm(\LL_n',\tilde\T)$ is the orthogonal complement of $\F^\pm(\LL_\infty')\otimes\Qp/\Zp$ under the pairing
\[
\varprojlim H^1(\LL_n',\tilde\T)\times\varinjlim H^1(\LL_n',\F[p^\infty])\to\KK/\OO_\KK,
\]
where we identify $\F^\pm(\LL_\infty')\otimes\Qp/\Zp$ as a subgroup of $\varinjlim H^1(\LL_n',\F[p^\infty])=H^1(\LL_\infty',\F[p^\infty])$ under the Kummer map.

It can be checked directly that 
\[
\F^\pm(\KK'_n)=\F^\pm(\LL'_n)^{\Gal(\LL_n'/\KK_n')}
\]
since $\Gal(\LL_n'/\KK_n')\simeq \Gal(\LL_{n-1}'/\LL_{n-1}')$, and the trace map $\F(\KK_n')\to \F(\KK_{n-1}')$ restricted to $\F(\LL_n')$ agrees with the trace map $\F(\LL_n')\to \F(\LL_{n-1}')$.
Furthermore, as the degree of the extension $\LL_n'/\KK_n'$ is coprime to $p$, we have $H^1(\LL_n'/\KK_n',\F^\pm(\LL_n'))=0$. Therefore,
\[
\F^\pm(\KK'_n)\otimes\Qp/\Zp=\left(\F^\pm(\LL'_n)\otimes\Qp/\Zp\right)^{\Gal(\LL_n'/\KK_n')}, \quad H^1_\Iw(\KK'_\infty,\tilde\T)=H^1_\Iw(\LL'_\infty,\tilde\T)_{\Gal(\LL_\infty'/\KK_\infty')}.
\]
The natural commutative diagram
\[
\begin{tikzcd}
H^1_\Iw(\LL'_\infty,\tilde\T) \arrow[r, "\Col_\LL^\pm"] \arrow[d, ""'] & \Lambda^{\OO_\KK}[\Gal(\LL/\KK)] \arrow[d, ""] \\
H^1_\Iw(\KK'_\infty,\tilde\T) \arrow[r, "\Col_\KK^\pm"'] & \Lambda^{\OO_\KK}
\end{tikzcd}
\]
tells us that $\ker(\Col_\KK^\pm)=\ker(\Col_\LL^\pm)_{\Gal(\LL_\infty'/\KK_\infty')}$, which is equal to the orthogonal complement of $\left(\F^\pm(\LL_\infty')\otimes\Qp/\Zp\right)^{\Gal(\LL_\infty'/\KK_\infty')}=\F^\pm(\KK_\infty')\otimes\Qp/\Zp$ under the natural pairing
\[
H^1_\Iw(\KK'_\infty,\tilde\T)\times H^1(\KK'_\infty,\F[p^\infty])\to \KK/\OO_\KK.
\]
\end{remark}

\section{Formal groups associated with CM elliptic curves}\label{sec:CM-formal}
\subsection{Setup} \label{sec:setup} 
Let $E/\QQ$ be an elliptic curve, and let $p\ge 5$ be a fixed rational prime. Suppose that $E$ has potentially supersingular reduction at $p$ over $\QQ$.
Then the reduction type over $K$ is either good supersingular or bad additive. Following Serre--Tate \cite[\S2]{SerreTate} and Serre \cite[\S5.6]{SerreInventiones1972}, we define the following quantity measuring the failure of semistability:
\begin{equation}\label{eq:e}
 e\colonequals \begin{cases}
    1 & \text{if $E/\QQ$ has good supersingular reduction at $p$}, \\
    \displaystyle\frac{12}{\gcd(12, \ord_p(\Delta_E))} & \text{if $E/\QQ$ has bad additive reduction at $p$,}
\end{cases}   
\end{equation}
where $\Delta_E$ denotes the discriminant of $E$.
By definition, $e$ is a divisor of $12$, and more precisely, $e\in\{2,3,4,6\}$ \cite[\S5.6(a$_1$)]{SerreInventiones1972}.

We also have $e\mid p+1$: indeed, for $e=2$, this is clear, and for $e>2$, this is \cite[Lemma~2.1]{DelbourgoJNT}. Moreover, $E$ acquires good supersingular reduction over $\QQ_p(\pi_e)$, where $\pi_e$ is a fixed $e$-th root of $-p$, as explained in \cite[\S5.6(a$_1$)]{SerreInventiones1972} and \cite[pp.~51--53]{DelbourgoJNT}.
Note that the extension $\QQ(\pi_e)/\QQ$ is non-Galois when $e>2$, but the extension $\QQ(\mu_e,\pi_e)/\QQ$ is. We write $K\colonequals \QQ(\mu_e)$ if $e>2$ and $K\colonequals\QQ(\sqrt{-1})$ if $e=2$. Let 
\begin{equation} \label{eq:L-def}
    L\colonequals K(\pi_e)= \begin{cases} \QQ(\mu_e,\pi_e) & \text{if $e\in\{3,4,6\}$,} \\ \QQ(\sqrt{-1},\sqrt{-p}) & \text{if $e=2$.}\end{cases}
\end{equation} 
Then $p$ is unramified in $K/\QQ$, and the $p$-adic places in $K$ are totally tamely ramified in $L/K$. Also note that $L/\QQ$ is a dihedral Galois extension.

Suppose that $E$ has complex multiplication by the imaginary quadratic number field $K$.
Observe that as a consequence of the Shimura--Taniyama formula \cite[Lemme~5]{TateBourbaki}, $p$ is inert in $K$. Therefore there is a unique $p$-adic place in $K$ resp. $L$, and we denote the localisations by $\QQ_{p^2}$ resp. $\LL$.

\begin{lemma} \label{tor-ab}
    The extension $L(E_{\mathrm{tor}})/K$ is abelian.
\end{lemma}
\begin{proof}
    Indeed, $L$ is contained in a field $K(E[N])$ obtained from $K$ by adjoining points of order $N$ for some $N\ge3$ coprime to $p$: this follows from Theorem~2 and its Corollaries~2 and 3 in \cite{SerreTate}. In particular, $L(E_{\mathrm{tor}})$ is contained in $K(E_{\mathrm{tor}})$, which is abelian over $K$ by the main theorem of complex multiplication.
\end{proof}

Viewing $E$ as an elliptic curve over $L$ with complex multiplication by $K$, let $\psi$ denote the associated Hecke character of $L$. By \cite[Theorem~7.44]{ShimuraAutomorphic}, the abelianity of $L(E_{\mathrm{tor}})/K$ is equivalent to this factoring through $K$, that is, the existence of a Hecke character $\varphi$ of $K$ such that 
\begin{equation} \label{eq:psi-phi}
    \psi=\varphi\circ N_{L/K}.
\end{equation}

Working over the Hilbert class field $H$ of $K$, Sairaiji \cite{SairaijiRM}  showed that the formal group coming from formal completion and a formal group defined via $L$-factors have the same Honda type. Below we adapt some of these results to $L/K$: we will use relaxed Honda types as described in \S\ref{sec:relaxed-Honda}, and obtain an identification with the formal group constructed in \S\ref{sec:points}, which in turn will allow us to utilize plus/minus theory later on.

\begin{remark}
    We expect that the setup described above can be relaxed in various directions. For instance, we may consider elliptic curves not defined over $\QQ$ as follows. Let $K$ be a quadratic imaginary field in which $p$ is inert, and let $E/K$ an elliptic curve that has potentially supersingular reduction at the unique prime above $p$ in $K$. Define $e$ as above, depending on the reduction type of $E/K$ at the unique prime above $p$. Making the additional assumption that $e\mid p+1$, our results over $K$ become applicable to such elliptic curves with complex multiplication by an order in $\OO_K$. It is not clear whether having potentially supersingular reduction implies this divisibility for such elliptic curves: indeed, the proof of \cite[\S1.3]{delbourgoComp} breaks down in this case.
\end{remark}

\subsection{Cocycles}
Write $G\colonequals \Gal(L/K)$. For each $\sigma\in G$, fix an isogeny $\phi_\sigma: {}^\sigma E\to E$. We define a cocycle
\[c: G\times G\longrightarrow K^\times, \quad c(\sigma,\tau)\colonequals \phi_\sigma\circ\phi_\tau\circ\phi^{-1}_{\sigma\tau}.\]
\Cref{tor-ab} together with \cite[Proposition~2.6]{SairaijiRM} shows that $c(\sigma,\tau)=c(\tau,\sigma)$ for all $\sigma,\tau\in G$. By \cite[Proposition~2.4]{SairaijiRM}, the cocycle $c\in H^2(G,\overline{\QQ}^\times)$ is a coboundary. Therefore there exists a $1$-cocycle $\beta:G\to\overline{\QQ}^\times$ such that for all $\sigma,\tau\in G$,
\begin{equation} \label{eq:c-beta-cocycle}
    c(\sigma,\tau)=\beta(\sigma)\beta(\tau)\beta(\sigma\tau)^{-1}.
\end{equation}

Fix an invariant differential $\omega_E$ of $E$, and define
\[\alpha: G\longrightarrow L^\times, \quad \alpha(\sigma)\colonequals\frac{\phi_\sigma^*(\omega_E)}{\sigma\omega_E},\]
so that
\begin{equation} \label{eq:c-alpha-cocycle}
    c(\sigma,\tau)=\alpha(\sigma)\cdot {}^\sigma\alpha(\tau)\cdot\alpha(\sigma\tau)^{-1}
\end{equation}
for all $\sigma,\tau\in G$.

\subsection{Complex multiplication}
Recall that $E$ has complex multiplication by $\OO_K$. We fix an isomorphism $[-]: \OO_K\to\End(E)$ satisfying 
\begin{equation} \label{eq:normalisation}
    [x]^* \omega_E=x\omega_E
\end{equation}
for all $x\in\OO_K$.

Let $s\in\II_K$ be an idèle of $K$, and let $\sigma_s\colonequals[s,K]\in\Gal(K^\ab/K)$ be the Galois automorphism associated with $s$ by Artin reciprocity. It follows from the main theorem of complex multiplication that there is a diagram
\begin{equation} \label{eq:CM-diagram}
    \begin{tikzcd}
	& {K/\aa} & E \\
	K & {K/s^{-1}\aa} & {{}^\sigma E} \\
	K & {K/\aa} & E
	\arrow["\upsilon", from=1-2, to=1-3]
	\arrow[from=1-2, to=2-2]
	\arrow["{\sigma_s}", from=1-3, to=2-3]
	\arrow[from=2-1, to=2-2]
	\arrow["{\gamma_s}"', from=2-1, to=3-1]
	\arrow[from=2-2, to=2-3]
	\arrow[from=2-2, to=3-2]
	\arrow["{\phi_{\sigma_s}}", from=2-3, to=3-3]
	\arrow[from=3-1, to=3-2]
	\arrow["\upsilon"', from=3-2, to=3-3]
    \end{tikzcd}
\end{equation}
Here $\aa\subset\OO_K$ is an ideal, and $\gamma_s\in K$ is the unique element such that multiplication by $\gamma_s$ makes the lower left square commute. By \cite[Proposition~2.10]{SairaijiRM}, the Hecke character $\varphi$ can be expressed as follows:
\begin{equation} \label{eq:phi-explicit}
    \varphi(s)=\frac{\gamma_s s_\infty^{-1}}{\beta(\sigma_s)}
\end{equation}
for all $s\in\II_K$.

\begin{lemma} \label{s-is-a-norm}
Let $s$ be the id\`ele with $p$ at the unique $p$-adic place and $1$ elsewhere.     Then $s\in N_{L/K}(\II_L)$.
\end{lemma}
\begin{proof}
    Recall that if $\KK$ is a local field of characteristic zero containing $\mu_e$, and $a,b\in \KK$, then the Hilbert symbol $(a,b)$ is $1$ if and only if $a$ is a norm from $\LL=\KK(\sqrt[e]{b})$ to $\KK$. For $\KK=\QQ_{p^2}(\mu_e)$, we have $(-1,-p)=1$ by \cite[V\S3, Proposition~3.4]{NeukirchANT}. Therefore, there exists some $u\in\QQ_{p^2}(\mu_e, \sqrt[e]{-p})$ whose norm is $-1$.

    Let $\tilde s$ be the $L$-id\`ele with 
    \[-u^{\frac{1+(-1)^e}{2}}\pi_e\] 
    at the unique $p$-adic place and $1$ elsewhere. The norm of $\tilde s$ has $1$ at the non-$p$-adic places, and at the $p$-adic place, it is $p$ because $N_{\LL/\KK}(-\pi_e)=(-1)^e\cdot (-p)$. So $N_{L/K}(\tilde s)=s$, as desired.
\end{proof}
Since $L$ is an abelian extension of $K$, the kernel of the Artin map relative to $L/K$ is $K^\times N_{L/K}(\II_L)$, and therefore \cref{s-is-a-norm} implies that $\sigma_s|_L=\id$. It follows that $\phi_\sigma=\id$. Then, by definition, we have $\alpha(\sigma_s)=1$. Using \eqref{eq:c-alpha-cocycle}, this implies $c(\sigma_s,\tau)=1$ for all $\tau\in G$. Now, from \eqref{eq:c-beta-cocycle} it follows that $\beta(\sigma_s)=1$. 
Using the bottom left square of the diagram as well as the chosen normalisation \eqref{eq:normalisation}, we find that $\gamma_s=p$.
From \eqref{eq:psi-phi} and \eqref{eq:phi-explicit} we conclude that
\[\psi(\tilde s)=\varphi(s)=\gamma_s=p.\]

\subsection{Formal groups} \label{sec:formal-groups}
Let $\Ehat$ be the formal group of the elliptic curve $E$, obtained by taking a local parameter at the origin for a Weierstraß equation of $E$. Let $f$ be the logarithm of $\Ehat$. As in the proof of \cite[Proposition~3.2]{SairaijiRM}, the diagram \eqref{eq:CM-diagram} implies a congruence
\[f^{-1}(\alpha(\sigma_s) {}^{\sigma_s}f(X^q))\equiv f^{-1}(\gamma_sf(X)) \pmod{\pi\OO_\LL[[X]]},\]
for all $s\in\II_K$, where $q=p^e$ is the order of the residue field of $\LL$. By \cref{DemchenkoThm3}, $f$ has a Demchenko type, therefore, we can use \cref{Honda4.2} to deduce that
\[\alpha(\sigma_s) {}^{\sigma_s} f(X^q)\equiv \gamma_s f(X) \pmod{\pi\OO_\LL[[X]]}.\]
In particular, at the id\`ele $s$ with $p$ in the $p$-adic place and $1$ elsewhere, we find that $f$ has relaxed Honda type $X^2-p$ over $\OO_\LL$.

We also consider the function
\[\xi(s)\colonequals \sum_{\tau\in G} \frac{\alpha(\tau^{-1})\beta(\tau)}{c(\tau,\tau^{-1})}\sum_{[\aa,K]|_L=\tau}\frac{\varphi(\aa)}{\Norm(\aa)^s},\]
where $\aa$ runs through the ideals of $\OO_K$ coprime to bad primes of $E$ in $L$.
As explained in \cite[\S2.3]{SairaijiRM}, this is a linear combination of $L$-functions $L(\varphi\chi,s)$ where $\chi$ runs over characters of $G$.
Let $\G$ be the formal group attached to $\xi$, and let $g$ be its logarithm. By the same argument as in \cite[Proposition~3.5]{SairaijiRM} at the id\`ele $s$ considered above, we find that $g$ has relaxed Honda type $X^2-p$ over $\OO_\LL$.

\begin{proposition} \label{strong-iso-E-F}
    The formal groups $\Ehat$ and $\G$, and $\F$ (from \S\ref{sec:points}, with $\rho$ taken to be $p$) are strongly isomorphic over $\OO_\LL$.
\end{proposition}
\begin{proof}
    By the above discussion and \cref{ell-Honda-type}, all three formal groups share the same relaxed Honda type. The assertion now follows from \cref{HondaThm2}. 
\end{proof}

\section{\texorpdfstring{$p$}{p}-adic \texorpdfstring{$L$}{L}-functions and Iwasawa main conjectures}\label{sec:IMC}

\subsection{The Kummer pairing}
Retain the assumptions of \S\ref{sec:setup}, so that $E$ is an elliptic curve with complex multiplication by $K$, $L= K(\pi_e)$, and $e\mid(p+1)$. Let $\sK$ denote the field $K(E[p^\infty])$. Let $\Phi:G_K\to\OO_\KK^\times$ be the character that describes the action of $G_K$ on $T_p(E)$. By \cref{strong-iso-E-F}, the restriction of $\Phi$ to $G_{K_p}$ coincides with the Lubin--Tate character $\chi_\F$ from \S\ref{sec:points}.
\begin{lemma}
    We have $L\subset \sK$, and there is a unique $p$-adic prime in $\sK$. 
\end{lemma}
\begin{proof}
   Let $\overline\Phi:G_{K}\to\FF_{p^2}^\times$ be the reduction of $ \Phi$ modulo~$p$. Let $\chi:G_{ K}\to\mu_e$ be the composition of $\overline\Phi$ with the natural projection $\FF_{p^2}^\times\twoheadrightarrow\mu_e$, where we use $e\mid(p+1)$. Then $\chi$ factors through $K(E[p])$ since $\overline \Phi$ does. Therefore, the fixed field $\overline{K}^{\ker \chi}$ of $\chi$ is a degree $e$ cyclic extension of $K$ contained in $\sK$. By Kummer theory, we have $\overline { K}^{\ker \chi}=K(\sqrt[e]{u})$ for some $u\in  K^\times$. It follows from ramification theory that $u$ has $p$-adic valuation 1 in $ K$, and hence agrees (possibly up to sign and $e$-th powers) with $-p$. Hence $\overline { K}^{\ker \chi}=L$, and in particular, $L\subset\sK$, as claimed.

    Now we may use the argument employed by Pollack and Rubin \cite[p.~450]{PollackRubin} for the Galois representation $\Phi|_{G_L}$. Using the description of the formal group in \S\ref{sec:formal-groups}, this shows that $\Phi|_{G_L}$ is surjective on the inertia subgroup, so there must be a unique prime above $\sqrt[e]{-p}$ in $\sK$.
\end{proof}

Let $K_\infty\subset\sK$ denote the $\Zp^2$-extension of $K$.
Further, define $\Lambda(K_\infty)$ as the 2-variable Iwasawa algebra $\Zp[[\Gal( K_\infty/K)]]$, and write $\Lambda^{\OO_\KK}( K_\infty)=\Lambda( K_\infty)\otimes_{\Zp}\OO_\KK$. 
Let $ K_\cyc$ be the cyclotomic $\Zp$-extension of $ K$. We define the Iwasawa algebras $\Lambda( K_\cyc)$ and $\Lambda^{\OO_\KK}( K_\cyc)$ similarly. 

 Given an intermediate extension $ K\subset F\subset \sK$ and a $\Lambda(\sK)$-module $Y$, we define
\[
Y(\Phi^{-1})\colonequals Y\otimes \Hom_{\OO_{ K}}(E[p^\infty],\KK/\OO_\KK) \, \text{ and } \, Y_F^{ \Phi}\colonequals Y(\Phi^{-1})/\langle \sigma-1:\sigma\in\Gal(\sK/F)\rangle.
\]

For a finite abelian extension $F$ of $ K$, we define $U_F$ as the pro-$p$ part of the local units of $(\OO_F\otimes\Zp)^\times$.  Define
\[
 \cU=\varprojlim U_F,
\]
where the inverse limits are taken over finite extensions $F$ of $ K$ inside $\sK$ and the connecting maps are norm maps.

By an abuse of notation, we write $p$ for the unique prime of $\sK$ lying above $p$. Let 
\[
\langle\sim,\sim\rangle:\left(E(\sK{}_{,p})\otimes\Qp/\Zp\right)\times \cU(\Phi^{-1})\longrightarrow \KK/\OO_\KK
\]
be the $\OO_\KK$-linear Kummer pairing defined as in \cite[(3)]{PollackRubin}. Furthermore, as in Proposition~4.1 of \textit{op.~cit.}, we have
\begin{equation}\label{eq:iso-local}
\cU_{ K_\cyc}^\Phi\cong\Hom_\OO(E( K_{\cyc,p})\otimes\Qp/\Zp,\KK/\OO_\KK).    
\end{equation}

Note that the Kummer map gives $E( K_{\cyc,p})\otimes\Qp/\Zp=H^1(K_{\cyc,p},E[p^\infty])$. Taking limits in \eqref{eq:KP-pairing} and identifying the formal group of $E$ at $p$ with $\F$ by \cref{strong-iso-E-F} gives a perfect pairing
\[
H^1_\Iw(K_{\cyc,p},\tilde\T)\times H^1(K_{\cyc,p},E[p^\infty])\to \KK/\OO_\KK.
\]
Thus, comparing this with \eqref{eq:iso-local} gives an isomorphism
\begin{equation}
    \cU_{ K_\cyc}^\Phi\cong H^1_\Iw(K_{\cyc,p},\tilde\T).\label{eq:U-HIw}
\end{equation}
From now on, we assume that \eqref{eq:e-ne-p+1} holds (with $e$ given by \eqref{eq:e}). Recall from \cref{rk:ep+1} that this only excludes the case where $(p,e)=(5,6)$.

 Through \eqref{eq:U-HIw}, we can regard the Coleman maps $\Col_\KK^\pm$ as maps on $\cU_{ K_\cyc}^\Phi$.

\subsection{Selmer groups}
\label{sec:Sel}
Let $\sM$ (resp. $\sL$) denote the maximal abelian $p$-extension of $\tsK$ that is unramified outside $p$ (resp. unramified everywhere), and let $\cX\colonequals\Gal(\sM/\tsK)$ and $\cA\colonequals\Gal(\sL/\tsK)$.

\begin{theorem}
    We have the following isomorphism
    \[
    \Sel_{p^\infty}(E/ K_\cyc)\cong\Hom_{\OO_{ K}}(\cX_{ K_\cyc}^{\Phi},\KK/\OO_\KK).
    \]
\end{theorem}
\begin{proof}
As the action of $G_{\tsK}$ on $E[p^\infty]$ is trivial, we have
\[
H^1(\tsK,E[p^\infty])=\Hom(G_{\tsK},E[p^\infty]).
\]
Further, it follows from the definition of $\cX$ that
\[
\ker\left(H^1(\tsK,E[p^\infty])\to \prod_{v\nmid p}\frac{H^1(\tsK{}_v,E[p^\infty])}{H^1_\mathrm{ur}(\tsK{}_v,E[p^\infty])}\right)=\Hom(\cX,E[p^\infty]).
\]

The Selmer group $\Sel_{p^\infty}(E/ K_\cyc)$ is given by
\[
\ker\left(H^1(\tsK,E[p^\infty])\longrightarrow \prod\frac{H^1(\tsK{}_v,E[p^\infty])}{H^1_\f(\tsK{}_v,E[p^\infty])}\right),
\]
where the product runs over all primes of $\tsK$ and $H^1_\f(\tsK{}_v,E[p^\infty])$ denotes the Bloch--Kato subgroup. When $v\nmid p$, we have
\[
H^1_\mathrm{f}(\tsK{}_v,E[p^\infty])=H^1_\mathrm{ur}(\tsK{}_v,E[p^\infty])
\]
since the action of $G_{\tsK}$ on $E[p^\infty]$ is unramified. For $v\mid p$,
\[
\left(\frac{H^1(\tsK{}_v,E[p^\infty])}{H^1_\mathrm{f}(\tsK{}_v,E[p^\infty])}\right)^\vee\cong \varprojlim_{{\mathscr{K}}} H^1_\mathrm{f}({\mathscr{K}},T_p(E))=\varprojlim_{{\mathscr K}}\Ehat({\mathscr{K}}),
\]
where the inverse limits run over finite subextensions ${\mathscr K}$ of $\tsK{}_v/ K_p$ and the connecting maps are corestrictions and norm maps, respectively.
As $\Ehat$ has height $2$ by \cref{strong-iso-E-F}, it follows from \cite[Lemma~2.2]{rubin85} that these inverse limits are in fact zero. Hence, we deduce
\[
\Sel_{p^\infty}(E/ K_\cyc)=\Hom(\cX,E[p^\infty]).
\]
The proof of \cite[Proposition~1.2]{rubin85} can then be adapted to our setting to conclude the proof.
\end{proof}

In order to access the plus/minus theory, from now on, we assume that the integer $e$ given by \eqref{eq:e} is not equal to $p+1$ as in \eqref{eq:e-ne-p+1}.
\begin{definition}\label{def:pmSel}
After identifying $\F$ with $\Ehat$ and $K_p$ with $\KK$ through \cref{strong-iso-E-F}, we write
\[
H^1_\pm(K_{\cyc,p},E[p^\infty]):=\F^\pm(K_{\cyc,p})\otimes\Qp/\Zp\subset H^1(K_{\cyc,p},E[p^\infty]),
\]
and
\[
\Sel_{p^\infty}^\pm(E/ K_{\cyc}) \colonequals \ker\left(\Sel_{p^\infty}(E/ K_{\cyc})\longrightarrow \frac{H^1( K_{\cyc,p},E[p^\infty])}{H^1_\pm( K_{\cyc,p},E[p^\infty])}\right).
\]    
\end{definition}
It follows from Corollary~\ref{cor:cofree} that the $\Lambda^{\OO_\KK}( K_\cyc)$-modules 
$\Hom\left(H^1_\pm(K_{\cyc,p},E[p^\infty]), \QQ_p/\ZZ_p\right)$ are free of rank one.

Let $\widetilde\cV^\pm\subset \cU_{ K_\cyc}^{\Phi}$ be the subgroup corresponding to 
\[
\Hom_{\OO_{ K}}\left(\frac{H^1(K_{\cyc,p},E[p^\infty])}{H^1_\pm(K_{\cyc,p},E[p^\infty])},\KK/\OO_\KK\right)
\]
under \eqref{eq:iso-local}. Through \eqref{eq:U-HIw}, we can identify $\widetilde\cV^\pm$ with the kernel of $\Col_\KK^\pm$.

Let $\alpha:\cU\to\cX$ be the Artin map of global class field theory. Then
\begin{equation} \label{eq:PR4.3}
\Sel_{p^\infty}^\pm(E/ K_{\cyc,p})=\Hom_{\OO_{ K}}\left(\cU_{ K_\cyc}^{\Phi}\big/\alpha(\widetilde\cV^\pm),\KK/\OO_\KK\right)
\end{equation}
as in \cite[Theorem~4.3]{PollackRubin}.

\begin{proposition} \label{PR4.4}
    The following assertions hold.
    \begin{enumerate}
    \item[(i)] $\mathcal{U}_{ K_{\infty}}^{\Phi}$ is free of rank two over $\Lambda^{\mathcal{O}_\KK}( K_{\infty})$ and $\mathcal{U}_{ K_\cyc}^{\Phi}$ is free of rank two over $\Lambda^{\mathcal{O}_\KK}( K_\cyc)$.
    \item[(ii)] $\widetilde{\mathcal{V}}^{\pm}$ and $\mathcal{U}_{ K_\cyc}^{\Phi} / \widetilde{\mathcal{V}}^{\pm}$ are free of rank one over $\Lambda^{\mathcal{O}_\KK}$.
    \item[(iii)] There is a (noncanonical) submodule $\cV^{\pm} \subset \cU_{ K_{\infty}}^{\Phi}$ whose image in $\cU_{ K_\cyc}^{\Phi}$ is $\widetilde{\cV}^{\pm}$ and such that $\cV^{\pm}$ and $\cU_{ K_{\infty}}^{\Phi} / \cV^{\pm}$ are free of rank one over $\Lambda^{\mathcal{O}_\KK}( K_{\infty})$.
\end{enumerate}
\end{proposition}
\begin{proof}
    This follows from the same proof as \cite[Proposition~4.4]{PollackRubin}, with the input of \cite[Theorem~6.2]{Kobayashi} replaced by Corollary~\ref{cor:cofree}.
\end{proof}

\subsection{Elliptic units}
We define $E_F$ and $C_F$ to be the closure of the projection of the subgroup of global units and elliptic units in $U_F$, respectively. Define
\[
\cC\colonequals\varprojlim C_F,\quad \text{and}\quad \cE\colonequals\varprojlim E_F,
\]
where the inverse limits are taken over finite extensions $F$ of $K$ inside $\sK$ and the connecting maps are norm maps.

\begin{theorem} \label{PR5.1}
    The $\Lambda^{\OO_\KK}( K_\infty)$-module $\cC^{\Phi}_{ K_\infty}$ is free of rank one. It has a generator $\xi$ such that if $ K\subset F\subset  K_\infty$, $x\in E(F_p)$ and $\chi:\Gal(F/ K)\to\mu_{p^\infty}$ is a character, then
    \[
    \sum_{\sigma\in\Gal(F/ K)}\chi^{-1}(\sigma)\langle x^\sigma\otimes p^{-k},\xi\rangle=p^{-k}\frac{L(\varphi\chi,1)}{\Omega_E}\sum_{\sigma\in\Gal(F/ K)}\chi^{-1}(\sigma)\lambda_E(x^{\sigma}),
    \]
    where $\langle\sim,\sim\rangle$ is the Kummer pairing, $\varphi$ is the Hecke character of $K$ given in \eqref{eq:psi-phi}, $\Omega_E$ is the real Neron period of $E$, $\lambda_E$ is the logarithm map, and $F_p$ denotes the completion of $F$ at the unique $p$-adic prime.
\end{theorem}
\begin{proof}
    See \cite[Theorem~5.1]{PollackRubin}.
\end{proof}

\DeclareRobustCommand{\looongrightarrow}{%
  \DOTSB\relbar\joinrel\relbar\joinrel\relbar\joinrel\rightarrow
}

\begin{corollary} \label{PR5.2}
    \begin{enumerate}
        \item The natural map $\cC^{\Phi}_{ K_\cyc}\to\cU^{\Phi}_{ K_\cyc}$ is injective.
        \item The $\Lambda^{\OO_\KK}$-module $\cC^{\Phi}_{ K_\cyc}$ is free of rank one, and $\cC^{\Phi}_{ K_\cyc}\cap\widetilde \cV^\pm=0$.
        \item The $\Lambda^{\OO_\KK}$-module $\cE^{\Phi}_{ K_\infty}$ has rank one, and $\cE^{\Phi}_{ K_\infty}\cap\cV^\pm=0$.
    \end{enumerate}
\end{corollary}
\begin{proof}
    As in \cite[Corollary~5.2]{PollackRubin}, it suffices to show that the image of the generator $\xi$ given by \cref{PR5.1} under the natural map $\cU^{\Phi}_{ K_\infty}\to\cU^{\Phi}_{ K_\cyc}$ is not contained in $\widetilde\cV^\pm$. This can be verified by applying \cref{PR5.1} as follows. Let $L'_\infty$ denote the cyclotomic $\ZZ_p$-extension of $L$ with $n$th layer $L'_n$, so that the $p$-adic completion of $L'_n$ is $\LL'_n$. We will apply \cref{PR5.1} with $F\colonequals L_n$ and $x\colonequals d_n$ with $n$ and $\chi$ suitably chosen.
    
    Rohrlich's theorem \cite{Rohrlich} asserts that the $L$-value $L(\varphi\chi,1)$ is nonzero for almost all $\chi\in\Gal( K_\cyc/ K)^\wedge$, where we identify $\Gal( K_\cyc/ K)$ with a subgroup of $\Gal(K_\cyc/K)$ via the natural map. By Rohrlich's theorem, if $n$ is large enough, we may choose a character $\chi\in\Gal(L'_n/ K)^\wedge$ such that $L(\varphi\chi,1)\ne0$ and $\chi$ is induced by a character $\overline\chi\in \Gal(L'_n/L)^\wedge$. Applying \cref{KP2.9} and using that $\chi$ is induced by $\overline \chi$ as well as the fact that $\Gal(\LL/\KK)$ acts trivially on $d_n$ determines the sum
    \[\sum_{\sigma\in\Gal(F/K)}\chi^{-1}(\sigma) \lambda_E(d_n^{\sigma}).\]
    The proof then proceeds in the same way as in \cite{PollackRubin}.
\end{proof}

\begin{definition}\label{def:padicL}
    We define the plus and minus $p$-adic $L$-functions attached to $E$, denoted by $L_p^\pm(E,K_\cyc)$, as the images of $\xi$ under 
    \[
    \cC^{\Phi}_{ K_\cyc}\longrightarrow\cU^{\Phi}_{ K_\cyc}\stackrel{\Col_\KK^\pm}{\looongrightarrow}\Lambda^{\OO_\KK}.
    \]
\end{definition}

Let $g^\pm$ be the image of the generator $\xi$ of $\cC^\Phi_{ K_\infty}$ under the composite map
\[ \cC^{\Phi}_{ K_\infty} \twoheadrightarrow \cC^{\Phi}_{ K_\cyc} \hookrightarrow  \cU_{ K_\cyc}^\Phi\xrightarrow{\sim}\Hom_{\OO_\KK}(E( K_{\cyc,p})\otimes\Qp/\Zp,\KK/\OO_\KK),\]
where the first arrow is the natural map, the second arrow is \cref{PR5.2}(1), and the isomorphism is \eqref{eq:iso-local}.
Recall from \cref{cor:cofree} that $\Hom\left(\Ehat^\pm(\LL_\infty')\otimes \QQ_p/\ZZ_p, \QQ_p/\ZZ_p\right)$ is a free $\Lambda$-module of rank one, with a distinguished generator $f^\pm$. It follows that $g^\pm=f^\pm\cdot h^\pm$ for some $h^\pm\in\Lambda^{\OO_\KK}$.
By the definition of $\widetilde\cV^\pm$, we find that there is an isomorphism
\[\left.\cU^{\Phi}_{ K_\cyc} \middle/ \widetilde\cV^\pm \right. \simeq \Hom_{\OO_{ K}}\left(\Ehat^\pm( K_{\cyc,p})\otimes\QQ_p/\ZZ_p, \KK/\OO_\KK\right).\]
Consequently, we have an isomorphism of $\Lambda^{\OO_\KK}$-modules
\[
    \cU^{\Phi}_{ K_\cyc} \left.\middle/ \left(\widetilde\cV^\pm+\cC^{\Phi}_{ K_\cyc}\right)\right. \xlongrightarrow{\sim} \left.\Lambda^{\OO_\KK}\middle/h^\pm\Lambda^{\OO_\KK}.\right.
\]
Note that $h^\pm$ and $L_p^\pm(E,K_\cyc)$ differ by a unit in $\Lambda^{\OO_\KK}$. 

\begin{proposition} \label{PR7.2}
    For every $n\ge0$ and every character $\chi:\Gal(L_\infty'/ K)\to\mu_{p^n}$ of order $p^n$, the element $h^\pm$ satisfies the following interpolation formula:
    \[\frac{L(\varphi\chi,1)}{\Omega_E}\tau_{n+1}(\chi) = \chi(h^\pm) \chi(\omegat_n^\mp) ,\]
    where $\pm$ represents the parity of $n$.
  \end{proposition}
\begin{proof}
    The proof is analogous to \cite[Theorem~7.2]{PollackRubin}.
    By the definition of $h^\pm$, for every $n$ and $\chi$ as in the statement and $k\ge1$, we have
    \begin{equation} \label{eq:interpolation-phi}
        \sum_{\sigma\in\Gal(\LL'_n/\KK)} \chi(\sigma) \cdot \varphi^\pm\!\left(\sigma(d_n^\pm)\otimes p^k\right) = \chi(h^\pm) \sum_{\sigma\in\Gal(\LL'_n/\KK)} \chi(\sigma) \cdot f^\pm\!\left(\sigma(d_n^\pm)\otimes p^{-k}\right).
    \end{equation}
    Applying \cref{PR5.1} and \cref{KP2.9}, the left hand side of \eqref{eq:interpolation-phi} becomes
    \[e\cdot p^{-k}\cdot\frac{L(\varphi\chi,1)}{\Omega_E}\cdot \left(\frac{\rho}{p}\right)^{\left\lfloor\frac{n+1}{2}\right\rfloor}\tau_{n+1}(\chi)=e\cdot p^{-k}\cdot\frac{L(\varphi\chi,1)}{\Omega_E}\tau_{n+1}(\chi)\]
    as $\rho$ is taken to be $p$.
    The factor $e$ here is explained as follows. In \eqref{eq:interpolation-phi}, we are summing over $\Gal(\LL'_n/\KK)$, whereas the statement of \cref{KP2.9} involves summation over $\Gal(\LL'_n/\LL)$. Noting that $\Gal(\LL'_n/\KK)\simeq \Gal(\LL'_n/\LL)\times\Gal(\LL/\KK)$ and that $[\LL:\KK]=e$ is coprime to $p$, it follows that $\chi$ is trivial on $\Gal(\LL/\KK)$, hence the factor $e$.
    
    Meanwhile, applying \cref{cor:cofree} to the right hand side of \eqref{eq:interpolation-phi} yields
    \[e\cdot p^{-k}\cdot \chi(h^\pm)\cdot\chi(\omegat_n^\mp).\]
    The claim follows.
\end{proof}

\Cref{PR7.2} allows $\chi$ to be trivial, as opposed to the proof in \cite{PollackRubin}. This difference traces back to the discrepancy between the statements of our \cref{KP2.9} versus \cite[Theorem~3.2(iv)]{PollackRubin}, with our version being analogous to \cite[Proposition~2.9]{KimPark}.

\begin{corollary}\label{cor:rationalpadicL}
    The elements $h^\pm$ belong to $\Lambda=\Lambda(\Gamma)=\Zp[[X]]$.
\end{corollary}
\begin{proof}
\cref{PR7.2} tells us that for $\bullet\in\{+,-\}$, there are infinitely many finite-order characters $\chi$ of $\Gamma$ such that $\chi(h^\bullet)\in \Zp[\mu_{p^\infty}]$. Hence, $h^\bullet\in\Zp[[\Gamma]]$.
\end{proof}

We conclude with the proof of \cref{thmA}:
    
\begin{corollary}
\label{cor:IMC}
The $\Lambda^{\OO_\KK}$-module $\Hom_{\OO_{\KK}}(\Sel^\pm(E/ K_\cyc),\KK/\OO_\KK)$ is torsion. Furthermore, its characteristic ideal is generated by  $ L_p^\pm(E,K_\cyc) $.
\end{corollary}
\begin{proof}
    We have that $\cX^\Phi_{K_\cyc}/\alpha(\cV^\pm)$ and $\cU^\Phi_{K_\cyc}/(\cV^\pm+\cC^\Phi_{K_\cyc})$ are finitely generated torsion $\Lambda^{\OO_\KK}$-modules, and their characteristic ideals agree: this is can be shown in the same way as in \cite[Theorem~6.3]{PollackRubin}, using \cref{PR4.4,PR5.2}. Combining this with \eqref{eq:PR4.3} and \cref{PR7.2} shows the claim.
\end{proof}

\subsection{Descent from \texorpdfstring{$K_\cyc$}{Kcyc} to \texorpdfstring{$\QQ_\cyc$}{Qcyc}}
Let $\Qcyc$ denote the cyclotomic $\Zp$-extension of $\QQ$. Analogous to Definition~\ref{def:pmSel}, we define the plus and minus local conditions at $p$ and the corresponding Selmer groups:

\begin{definition}
We define 
\[
H^1_\pm(\Qcycp,E[p^\infty]):=\Ehat^\pm(\Qcycp)\otimes\Qp/\Zp
\]
and
\[
\Sel_{p^\infty}^\pm(E/ \Qcyc) \colonequals \ker\left(\Sel_{p^\infty}(E/ \Qcyc)\longrightarrow \frac{H^1( \Qcycp,E[p^\infty])}{H^1_\pm( \Qcycp,E[p^\infty])}\right).
\]    
\end{definition}
By the same reasoning as in \cref{rk:base-change}, 
\[
\Ehat^\pm(\Qcycp)\otimes\Qp/\Zp= \left(\Ehat^\pm(K_{\cyc,p})\otimes\Qp/\Zp\right)^{\Gal(K_{\cyc,p}/\Qcycp)}.
\]
Note that as $p\ne 2$, we have
\[
\Sel_{p^\infty}^\pm(E/ \Qcyc)\simeq\Sel_{p^\infty}^\pm(E/ K_\cyc)^{\Gal(K_\cyc/\Qcyc)}=(1+\iota)\Sel_{p^\infty}^\pm(E/ K_\cyc),
\]
where $\iota$ denotes the non-trivial element of $\Gal(K_\cyc/\Qcyc)$.
Furthermore,
\[
\Sel_{p^\infty}^\pm(E/ K_\cyc)=\Sel_{p^\infty}^\pm(E/ \Qcyc)\otimes_{\Zp}\OO_\KK.
\]
Thus,
\begin{equation}
\Hom_{\OO_\KK}(\Sel_{p^\infty}^\pm(E/ K_\cyc),\KK/\OO_\KK)=\Hom_{\Zp}(\Sel_{p^\infty}^\pm(E/ \Qcyc),\Qp/\Zp)\otimes_{\Zp}\OO_\KK.
\label{eq:base-change}    
\end{equation}

We write $\Lambda_\cyc$ for the Iwasawa algebra $\Zp[[\Gal(\Qcyc/\QQ)]]$. We can identify $\Gamma=\Gal(K_\cyc/K)$ with the Galois group $\Gal(\Qcyc/\QQ)$ and $\Lambda=\Lambda(\Gamma)$ with $\Lambda_\cyc$. Recall from \cref{cor:rationalpadicL} that $h^\pm\in\Lambda$. This allows us to make the following definition:
\begin{definition}
    We define $L_p^\pm(E/\Qcyc)$ as the image of $h^\pm$ under the natural map $\Lambda\simeq \Lambda_\cyc$.
\end{definition}
 As $L(E,s)=L(\varphi,s)$, we have the following interpolation formulae. Let $\chi$ be a character of $\Gal(\Qcyc/\QQ)$ of order $p^n>1$. Then
\[
\chi\left(L_p^\pm(E/\Qcyc)\right)=\frac{L(E,\chi,1)}{\Omega_E}\cdot\frac{\tau_{n+1}(\chi)}{\chi(\omegat_n^\mp)},
\]
where $\pm$ represents the parity of $n$.

Combining \eqref{eq:base-change} and \cref{cor:IMC}, we deduce the following Iwasawa main conjectures.

\begin{theorem}\label{thm:IMC-Q}
Let $K$ be an imaginary quadratic field. Let $p\ge5$ be a prime number that is inert in $K$. Let $E/K$ be an elliptic curve with complex multiplication by an order in $\OO_K$ satisfying \eqref{eq:L-def} with $(p,e)\ne (5,6)$.
Assume that $E$ has potentially supersingular reduction at $p$ with $e\ne p+1$.
    The $\Lambda_\cyc$-module $\Hom_{\Zp}(\Sel^\pm(E/ \Qcyc),\Qp/\Zp)$ is torsion. Furthermore, its characteristic ideal is generated by  $ L_p^\pm(E,\Qcyc) $. \qed
\end{theorem}

\section{Growth of Tate--Shafarevich groups} \label{sec:growth-of-Sha}
We explain how the Selmer groups defined in \S\ref{sec:Sel} can be used to obtain formulae of the growth of Tate--Shafarevich groups, similar to \cite[Theorem~1.4]{Kobayashi}. As in \S\ref{sec:IMC}, $K_\cyc/K$ denotes the cyclotomic $\Zp$-extension. 
Let $K_n$ be the unique sub-extension of degree $p^n$.
The calculations in the aforementioned work can be mostly generalized to our current setting, with the key new input being the structure of the plus and minus subgroups studied in \S\ref{sec:points}. 

Let $\Sigma$ be a finite set of places of $K$ containing the unique $p$-adic place, the unique archimedean place, the places ramifying in $K/\QQ$, and the places at which $E$ has bad reduction. Let $K_\Sigma$ denote the maximal $\Sigma$-ramified extension of $K$. For any intermediate field $F$ of $K_\Sigma/K$, let $H^i_\Sigma(F,\sim)\colonequals H^i(K_\Sigma/F,\sim)$ denote the $\Sigma$-ramified Galois cohomology groups.

We define $H^i_{\Iw,\Sigma}(K_{\cyc},\sim)=\varprojlim H^i_\Sigma(K_n,\sim)$. 
Recall that $\KK$ is the localisation of $K$ at the unique prime above $p$, and $\Phi:G_K\to\OO_\KK^\times$ is the character that describes the action of $G_K$ on $T_p(E)$. 
We define $\tilde\Phi$ as $\Phi^{-1}\chi_\cyc$ as in \S\ref{sec:Coleman-maps}. 
By an abuse of notation, we write 
\[
\tilde\T=\OO_\KK(\tilde\Phi)=\Hom_{\OO_\KK}(T_p(E),\OO_\KK)(1).
\]
which is an $\OO_\KK[G_K]$-module. 
If we restrict it to $G_{K_p}$, we recover $\OO_\KK(\tilde\chi_\F)$ considered in \S\ref{sec:Coleman-maps}.
When $\Sigma$ is the set of bad primes for our CM elliptic curve, it follows from the same proof as \cite[\S15.15]{kato04} that $H^1_{\Iw,\Sigma}(K_{\cyc},\tilde \T)$ is a free $\Lambda^{\OO_\KK}$-module of rank one, whereas $H^2_{\Iw,\Sigma}(K_{\cyc},\tilde \T)$ is torsion over $\Lambda^{\OO_\KK}$.

\begin{definition} \label{def:Sel}
    Let $F$ be either $K_n$ or $K_\cyc$. We define
    \begin{align*}
        \Sel^0_{p^\infty}(E/F) &\colonequals \ker\left(H^1_\Sigma(F,E[p^\infty])\to \prod_{v}H^1(F_v,E[p^\infty])\right),\\
        \Sel_{p^\infty}(E/F) &\colonequals \ker\left(H^1_\Sigma(F,E[p^\infty])\to  \frac{H^1(F_v,E[p^\infty])}{E(F_p)\otimes\Qp/\Zp}\times\prod_{v\nmid p}H^1(F_v,E[p^\infty])\right),\\
        \Sel_{p^\infty}^\pm(E/F) &\colonequals \ker\left(H^1_\Sigma(F,E[p^\infty])\to  \frac{H^1(F_v,E[p^\infty])}{H^1_\pm(F_v,E[p^\infty])}\times\prod_{v\nmid p}H^1(F_v,E[p^\infty])\right),
    \end{align*}
       where the products run over places of $F$ that divide those in $\Sigma$, and $H^1_\pm(K_{n,p},E[p^\infty])$ denotes the image of $\Ehat^\pm(K_{n,p})\otimes\Qp/\Zp$ under the Kummer map.
We write $\cX^\bullet(E/F)$ for the Pontryagin dual $\Hom_{\OO_\KK}\left(\Sel_{p^\infty}^\bullet(E/F),\KK/\OO_\KK\right)$, where $\bullet\in\{0,\emptyset,+,-\}$.
    \end{definition}

By Corollary~\ref{cor:IMC}, $\cX^\pm(E/K_\infty)$ is a $\Lambda^{\OO_\KK}$-torsion module, and so the same is true for  $\cX^0(E/K_\infty)$.
Recall that $H^1_\pm(K_{n,p},\tilde \T)$ denotes the orthogonal complement of $\Ehat^\pm(K_{n,p})\otimes\Qp/\Zp$ under the pairing given by \eqref{eq:KP-pairing}.

\begin{definition}
    For $n\ge0$, define 
    \begin{align*}
    \cY(E/K_n) &\colonequals \coker\left( H^1_\Sigma(K_n,\tilde\T)\to H^1(K_{n,p},\tilde\T)\to \frac{H^1(K_{n,p},\tilde\T)}{H_\f^1(K_{n,p},\tilde\T)}\right),\\
    \cY'(E/K_n) &\colonequals \coker\left(H^1_{\Iw,\Sigma}(K_{\cyc},\tilde\T)\to H^1_\Iw(K_{\infty,p},\tilde\T) \to  \frac{H^1(K_{n,p},\tilde\T)}{H_\f^1(K_{n,p},\tilde\T)}\right),\\ 
    \cY''(E/K_n) &\colonequals \coker\left(H^1_{\Iw,\Sigma}(K_{\cyc},\tilde\T)\to H^1_\Iw(K_{\infty,p},\tilde\T) \to \frac{H^1(K_{n,p},\tilde\T)}{H_+^1(K_{n,p},\tilde\T)}\oplus \frac{H^1(K_{n,p},\tilde\T)}{H_-^1(K_{n,p},\tilde\T)}\right).
    \end{align*}
\end{definition}

\begin{lemma}\label{lem:PT}
    For every $n\ge0$, there are short exact sequences
    \begin{align*}
    H^1_\Sigma(K_n,\tilde\T) \to \frac{H^1(K_{n,p}, \tilde\T)}{H^1_\pm(K_{n,p}, \tilde\T)}&\to\cX^\pm(E/K_n)\to\cX^0(E/K_n)\to 0, \\
    0\to \cY(E/K_n)&\to\cX(E/K_n)\to\cX^0(E/K_n)\to 0.
    \end{align*}
\end{lemma}
\begin{proof}
  These are analogues of the sequences (7.17), (7.18), and (10.35) of \cite{Kobayashi}.
  We shall apply the Poitou--Tate exact sequence, as stated in \cite[Proposition~A.3.2]{PerrinRiou}, to $\T$ over $K_n$.
  For this, first recall from \S\ref{sec:Coleman-maps} that $G_{K_p}$ acts on $\T$ and $\tilde\T$ by $\chi_\F$ and $\tilde\chi_\F=\chi_\F^{-1}\chi_\cyc$, respectively. 
  Moreover, as explained in \cite[\S15.10]{kato04}, we have $T_p(E)=\begin{pmatrix} \varphi & 0 \\ 0 & \varphi^\iota \end{pmatrix}$, where $\varphi$ is the Hecke character attached to $E$, and $\iota\in G_\QQ$ denotes complex conjugation. The explicit description of the Galois action in \textit{loc.~cit.} implies $\varphi(\sigma)=\iota\varphi(\iota\sigma\iota)$ for all $\sigma\in G_\QQ$.
  Consequently, we have $\chi_\F^2=\chi_\cyc$, and therefore $\tilde\T\simeq\T^*(1)$, where $(\sim)^*=\Hom(\sim,\ZZ_p)$ denotes the $\ZZ_p$-dual.
  Furthermore, let $H^1_\pm(K_n,E[p^\infty])\subseteq H^1(K_n,E[p^\infty])$ denote the subset whose elements have image in $H^1_\pm(K_{n,p},E[p^\infty])$ upon restriction to $H^1(K_{n,p},E[p^\infty])$; there are no conditions at non-$p$-adic places.
  Now the Poitou--Tate sequence with and without plus/minus local conditions reads:
  \begin{align*}
      H^1_\Sigma(K_n,\tilde\T)\to \frac{H^1(K_{n,p}, \tilde\T)}{H^1_\pm(K_{n,p}, \tilde\T)} \to H^1_\pm(K_n,E[p^\infty])^\vee\to H^2_\Sigma(K_n,\tilde\T), \\
      H^1_\Sigma(K_n,\tilde\T)\to \frac{H^1(K_{n,p}, \tilde\T)}{E(K_{n,p})\otimes\Zp} \to H^1(K_n,E[p^\infty])^\vee\to H^2_\Sigma(K_n,\tilde\T).
  \end{align*}
  By the argument in \cite[p.11]{CoatesSujathaGalCohEllCurves}, the last term is $H^2_\Sigma(K_n,\tilde\T)=0$. The claims now follow from the definitions of $\cX$ and $\cY$.
\end{proof}

\begin{lemma}\label{lem:Y-prime}
    For every $n\ge0$, there is a short exact sequence
    \[0\longrightarrow \cY'(E/K_n)\longrightarrow \cY''(E/K_n)\longrightarrow \frac{H^1(K_{p},\tilde\T)}{H^1_\f(K_{p},\tilde\T)}\longrightarrow 0.\]
\end{lemma}
\begin{proof}
    In analogy with \cite[(10.39)]{Kobayashi}, this follows from the second short exact sequence of \cref{Kob8.12}  after taking Pontryagin duals and using the definitions of $\cY'$ and $\cY''$.
\end{proof}

As in \cite[\S10.2]{Kobayashi}, given a projective system of finitely generated $\Zp$-modules $(M_n)_{n\ge1}$, with connecting maps $\pi_n:M_n\to M_{n-1}$, we define the Kobayashi rank as
\[
\nabla M_n \colonequals \length_{\Zp} \left(\ker\pi_n\right) -\length_{\Zp}\left(\coker\pi_n\right)+\dim_{\Qp}\left(M_{n-1}\otimes\Qp\right)
\]
when both $\ker\pi_n$ and $\coker\pi_n$ are finite.
When $M$ is a finitely generated torsion $\Lambda$-module, we define $\nabla M_n$ as $\nabla (M/\omega_n M)$ whenever the latter is defined. We recall the following lemma:

\begin{lemma}\label{lem:kob-rank}
    The following assertions are valid:
    \begin{enumerate}[label=(\roman*)]
        \item Suppose that we have a short exact sequence of projective limits
        \[
        0 \longrightarrow (M_n')_{n\ge1} \longrightarrow (M_n)_{n\ge1} \longrightarrow (M_n'')_{n\ge1}
        \longrightarrow 0.
        \]
        If two of \(\nabla M_n\), \(\nabla M_n'\), \(\nabla M_n''\) are defined,
        then the other is also defined. In this case, we have
        \[
        \nabla M_n = \nabla M_n' + \nabla M_n''.
        \]
        
        \item Suppose that \(M_n\) are constant. If \(M_n\) is finite or the
        transition map
        \[
        M_n \longrightarrow M_{n-1}
        \]
        is given by multiplication by \(p\), then
        \[
        \nabla M_n = 0.
        \]
    \end{enumerate}
\end{lemma}
\begin{proof}
    See \cite[Lemma 10.4]{Kobayashi}.
\end{proof}

As before, we fix a compatible system of primitive $p$-power roots of unity $\zeta_{p^n}$ for $n\ge1$, satisfying $\zeta_{p^n}^p=\zeta_{p^{n-1}}$. We write $\epsilon_n=\zeta_{p^n}-1$.

\begin{lemma}\label{lem:evaluate}
    If $M=\Lambda^{\OO_\KK}/(f)$ for some $f\in \Lambda^{\OO_\KK}$, and $f(\epsilon_n)\ne0$ then
    \[
    \nabla M_n=2\ord_{\epsilon_n}f(\epsilon_n).
    \]
    For $n\gg0$, the Kobayashi rank $\nabla M_n$ is defined, with value given by
     \[
    \nabla M_n=2(\lambda(M)+\varphi(p^n)\mu(M)).
    \]
   \end{lemma}
\begin{proof}
    This follows from the same proof as \cite[Lemma 10.5]{Kobayashi}, taking into account the change in coefficients from $\Zp$ to $\OO_\KK$, which increases the lengths and dimensions by a factor of two.
\end{proof}

\begin{proposition}\label{pro:Y-X0}
    For $n\gg0$, we have
    \begin{enumerate}[label=(\roman*)]
        \item $\nabla\cY(E/K_n)=\nabla\cY'(E/K_n)=\nabla\cY''(E/K_n)-1$;
        \item $\nabla\cX^0(E/K_n)=2\ord_{\epsilon_n}f_0(\epsilon_n)$, where $f_0$ is the characteristic polynomial of $\cX^0(E/K_\cyc)$.
    \end{enumerate}
\end{proposition}
\begin{proof}
    This follows from the same proof as \cite[Proposition~10.6]{Kobayashi} using \cref{lem:Y-prime,lem:evaluate}.
\end{proof}

\begin{definition}
      Let $\bz\in H^1_{\Iw,\Sigma}(K_\cyc,\tilde \T)$. Define
      \[
      \Colu_n(\bz)=\omegat_n^+\Col^-(\bz)+\omegat_n^-\Col^+(\bz)\in\Lambda^{\OO_\KK}.
      \]
     \end{definition}

\begin{lemma}\label{lem:col-non-zero}
Let $\bz$ be a $\Lambda^{\OO_\KK}$-basis of $H^1_{\Iw,\Sigma}(K_\cyc,\tilde \T)$.   For $n\gg0$, we have $\Colu_n(\bz)(\epsilon_n)\ne0$.
\end{lemma}
\begin{proof}
    It follows from the first exact sequence in Lemma~\ref{lem:PT} that $\Col^\pm(\bz)\ne0$. Thus, we have $\Col^\pm(\bz)(\epsilon_n)\ne0$ for $n\gg0$. Furthermore, for each $n\ge1$, exactly one of  $\omegat_n^+(\epsilon_n)$ and $\omegat_n^-(\epsilon_n)$ is zero. Hence, the lemma follows.
\end{proof}

\begin{lemma}\label{lem:nabla-Y}
Let $\bz$ be a $\Lambda^{\OO_\KK}$-basis of $H^1_{\Iw,\Sigma}(K_\cyc,\tilde\T)$.  For $n\gg0$, $\nabla\cY(E/K_n)$ is defined, with value given by
  $2\ord_{\epsilon_n}\left(\Colu_n(\bz)(\epsilon_n)\right)$.
\end{lemma}
\begin{proof}
The Coleman maps induce the following isomorphisms
\[
\frac{H^1(K_{n,p},\tilde\T)}{H^1_\pm(K_{n,p},\tilde\T)}\cong \Lambda_{n}^{\OO_\KK,\pm}.
\]
Therefore, 
\[
\cY''(E/K_n)=\frac{\Lambda_n^{\OO_\KK,+}}{\left(\Col^+(\bz)\right)}\oplus \frac{\Lambda_n^{\OO_\KK,-}}{\left(\Col^-(\bz)\right)}.
\]
Combining this with the following commutative diagram with exact rows
    \[
\begin{tikzcd}[column sep=large, row sep=large]
0 \arrow[r]
& \mathbb \OO_\KK
\arrow[r, "G"]
\arrow[d, equal]
& \Lambda_n^{\OO_\KK,+} \oplus \Lambda_n^{\OO_\KK,-}
\arrow[r, "F"]
\arrow[d, "\mathrm{proj}"']
& \Lambda_n^{\OO_\KK}\arrow[r]
\arrow[d, "\mathrm{proj}"']
& \Lambda_n^{\OO_\KK}\big/({\omegat}_n^{+},{\omegat}_n^{-})
\arrow[r]
\arrow[d, "\mathrm{proj}"']
& 0
\\
0 \arrow[r]
& \OO_\KK
\arrow[r, "G"]
& \Lambda_{n-1}^{{\OO_\KK},+} \oplus \Lambda_{n-1}^{{\OO_\KK},-}
\arrow[r, "F"]
& \Lambda_{n-1}^{\OO_\KK}
\arrow[r]
& \Lambda_{n-1}^{\OO_\KK}\big/({\omegat}_n^{+},{\omegat}_n^{-})
\arrow[r]
& 0
\end{tikzcd}
\]
where $F:(f,g)\mapsto \omegat^-_nf+\omegat_n^+g$ and $G:a\mapsto (\omegat^+_na,-\omegat_n^-a)$, we deduce the following diagram, again with exact rows:
 \[
\begin{tikzcd}[column sep=3em, row sep=large]
0 \arrow[r]
& \mathbb \OO_\KK
\arrow[r]
\arrow[d, equal]
& \cY''(E/K_n)
\arrow[r]
\arrow[d, "\mathrm{proj}"']
& {\Lambda_{n}^{\OO_\KK}}\big/{\left(\Colu_n(\bz)\right)}\arrow[r]
\arrow[d, "\mathrm{proj}"']
& \Lambda_n^{\OO_\KK}\big/({\omegat}_n^{+},{\omegat}_n^{-})
\arrow[r]
\arrow[d, "\mathrm{proj}"']
& 0
\\
0 \arrow[r]
& \OO_\KK
\arrow[r]
& \cY''(E/K_{n-1})
\arrow[r]
& {\Lambda_{n-1}^{\OO_\KK}}\big/{\left(\Colu_n(\bz)\right)}
\arrow[r]
& \Lambda_{n-1}^{\OO_\KK}\big/({\omegat}_n^{+},{\omegat}_n^{-})
\arrow[r]
& 0
\end{tikzcd}
\]
The right-most vertical map is bijective, as proved in \cite[Lemma~10.7]{Kobayashi}. 
Thus, it follows from \cref{lem:kob-rank,lem:col-non-zero} that for $n\gg0$, $\nabla\cY''(E/K_n)$ is defined, with value given by
\[
\ord_{\epsilon_n}(\Colu_n(\bz)(\epsilon_n))+1.
\]
Hence, the lemma follows from Proposition~\ref{pro:Y-X0}(i).
\end{proof}

Let $\lambda^\pm$ and $\mu^\pm$ denote the Iwasawa invariants of $\cX^\pm(E/K_\cyc)$. Recall that this means that, by the structure theorem of Iwasawa algebras, there exists a pseudo-isomorphism of $\Lambda^{\OO_\KK}$-modules
\begin{equation} \label{eq:X-pm-structure-theorem}
    \cX^\pm(E/K_\cyc) \longrightarrow \bigoplus_i \OO_\KK[[T]]/(f_i(T))\oplus \bigoplus_{j}\OO_\KK[[T]]/p^{m_j}.
\end{equation}
Here we identify $\Lambda^{\OO_\KK}$ with $\OO_\KK[[T]]$ by fixing a non-canonical isomorphism between them. Note that there is no free part since $\cX^\pm(E/K_\cyc)$ is $\Lambda^{\OO_\KK}$-torsion by \cref{cor:IMC}. The exponents $m_i$ are positive integers, the $f_j(T)\in\OO_\KK$ are Weierstraß polynomials, and
\begin{equation} \label{eq:def-lambda-mu-pm}
    \lambda^\pm\colonequals \sum_i \deg f_i(T), \quad \mu^\pm\colonequals\sum_j m_j.
\end{equation}

\begin{proposition}\label{nabla-X}
    For $n\gg0$, the Kobayashi rank $\nabla\cX(E/K_n)$ is defined, with value given by
    \[
    \begin{cases}
      2(\deg\omegat_n^-+\lambda^++\varphi(p^n)\mu^+)  &n \text{ is even,}\\
      2(\deg\omegat_n^++\lambda^-+\varphi(p^n)\mu^-)  &n \text{ is odd.}
    \end{cases}
    \]
\end{proposition}
\begin{proof}
      For $\bullet\in\{+,-,0\}$, let $f_\bullet$ be the characteristic polynomial of $\cX^\bullet(E/K_\cyc)$. The first exact sequence of Lemma~\ref{lem:PT} gives the following short exact sequence after letting $n\to \infty$
    \[
    0\longrightarrow \frac{\Lambda^{\OO_\KK}}{\Col^\pm(\bz)}\longrightarrow\cX^\pm(E/K_\cyc)\longrightarrow \cX^0(E/K_\cyc)\longrightarrow 0.
    \]
    Thus, there exists a unit $u^\pm\in\Lambda^{\OO_\KK,\times}$ such that
    $$
    \Col^\pm(\bz)\cdot f_0=u^\pm\cdot f_\pm,
    $$
    and so 
    $$
    \ord_{\epsilon_n}\Col^\pm(\bz)(\epsilon_n)=\ord_{\epsilon_n}f_{\pm}(\epsilon_n)-\ord_{\epsilon_n}f_0(\epsilon_n)
    $$

     It follows from Lemma~\ref{lem:kob-rank} and the second exact sequence in Lemma~\ref{lem:PT} that $\nabla\cX(E/K_n)$ is defined, with value given by
    \[
    \nabla\cY(E/K_n)+\nabla\cX^0(E/K_n).
    \]
Recall that $\omegat^+_n(\epsilon_n)=0$ if $n$ is even and  $\omegat^-_n(\epsilon_n)=0$ if $n$ is odd. Combined with Lemma~\ref{lem:nabla-Y}, we deduce that
    \[
    \nabla\cX(E/K_n)=2\left(\ord_{\epsilon_n}(\omegat_n^{\mp})+\ord_{\epsilon_n}(f_\pm(\epsilon_n)\right),
    \]
    where the choice of sign depends on the parity of $n$, for $n\gg0$. Hence, the proposition follows.
\end{proof}
We are now ready to prove \cref{thmB}.
\begin{corollary}\label{cor:sha}
Assume that $\Sha(E/K_n)[p^\infty]$ is finite with cardinality $p^{e_n}$ for all $n\ge0$. Let $2r_\infty=\displaystyle\lim_{n\to\infty}\rank E(K_n)$. For $n\gg0$, we have
\[
e_n-e_{n-1}= \begin{cases}
      2(\deg\omegat_n^-+\lambda^+-r_{\infty}+\varphi(p^n)\mu^+)  &\text{if $n$ is even,}\\
      2(\deg\omegat_n^++\lambda^--r_{\infty}+\varphi(p^n)\mu^-)  &\text{if $n$ is odd.}
    \end{cases}
\]
\end{corollary}
\begin{proof}
    The assertion follows by combining the short exact sequence $$0\to E(K_n)\otimes\Qp/\Zp\to\Sel_{p^\infty}(E/K_n)\to\Sha(E/K_n)[p^\infty]\to0$$ with \cref{lem:kob-rank}(i) and \cref{nabla-X}.
\end{proof}

\begin{remark}
    A Kida formula is available: that is, one can, under appropriate assumptions, describe how the invariants $\lambda^\pm$ and $\mu^\pm$ associated with finite extensions of $L$ vary in finite $p$-extensions. Namely, a Kida formula was established in \cite[Theorem~6.3]{FM} for at most weakly ramified extensions (that is, for extensions whose $p$-adic places have trivial second higher ramification groups). 
    
    Our setup falls under that of \textit{op.~cit.} by considering the elliptic curve $E/\QQ$ as defined over $L$, over which it has good reduction at $p$. Note that the assumptions of \textit{op.~cit.} include a condition that $p$ be completely split in the field over which the elliptic curve is defined (which fails for $L$). However, this condition is in fact superfluous: the local arguments at supersingular primes (see Lemma~2.3 and \S3) are also valid over finite extensions of $\QQ_p$.
\end{remark}

\subsection{Descent from \texorpdfstring{$K_n$}{Kn} to \texorpdfstring{$\QQ_{(n)}$}{Qn}}
For an integer $n\ge1$, let $\Qn$ denote the unique sub-extension of $\Qcyc/\QQ$ of degree $p^n$. In this section, we write $E'$ for the quadratic twist of $E$ over $K$. Note that $E'$ and $E$ are $\QQ$-isogenous since
\[
L(E,s)=L(\varphi,s)=L(\varphi\cdot \epsilon_K\circ N_{K/\QQ},s)=L(E',s),
\]
where $\epsilon_K$ is the quadratic character attached to $K$ and $N_{K/\QQ}$ is the norm map. Indeed, if $\mathfrak{q}$ is a prime of $K$ that is coprime to the discriminant of $K$, then $\epsilon_K\circ N_{K/\QQ}(\mathfrak{q})=1$. If it divides the discriminant, then $\varphi(\mathfrak{q})=0$. Furthermore, the degree of the isogeny is coprime to $p$ given that $E[p]$ is an irreducible $G_\QQ$-representation by \cite[Proposition 12]{SerreInventiones1972}.

Recall that 
\[
\Sha(E/K_n)\simeq\Sha(E/\Qn)\oplus \Sha(E'/\Qn).
\]
As we assume $\Sha(E/K_n)[p^\infty]$ is finite, we have in particular
\[
\ord_p(\Sha(E/K_n)[p^\infty])=\ord_p(\Sha(E/\Qn)[p^\infty])+\ord_p(\Sha(E'/\Qn)[p^\infty]).
\]

\begin{lemma}\label{doks}
There exists a constant $\nu$ such that
\[
\ord_p(\Sha(E/\Qn)[p^\infty])=\ord_p(\Sha(E'/\Qn)[p^\infty])+\nu
\]
for $n\gg0$.
\end{lemma}
\begin{proof}
    Let $\Omega_E$ and $\Omega_{E'}$ denote the real Neron periods of $E$ and $E'$, respectively.
    Note that $\ord_p(\Omega_E/\Omega_{E'})=\ord_p(\Delta_E/\Delta_{E'})=0$ since $E$ and $E'$ differ by a quadratic twist with discriminant coprime to $p$ (see the main result of \cite{pal}). Therefore, the lemma is a special case of \cite[Theorem~1.1]{Dokchitser}.
\end{proof}

\begin{corollary}\label{cor:sha-Q}
  Let $|\Sha(E/\Qn)[p^\infty]|=p^{e_n^\circ}$ for all $n\ge0$. For $n\gg0$, we have
\[
e^\circ_n-e^\circ_{n-1}= \begin{cases}
      \deg\omegat_n^-+\lambda^+-r_{\infty}+\varphi(p^n)\mu^+  &\text{if $n$ is even,}\\
      \deg\omegat_n^++\lambda^--r_{\infty}+\varphi(p^n)\mu^-  &\text{if $n$ is odd.}
    \end{cases}
\]
\end{corollary}
\begin{proof}
    Let $|\Sha(E'/\Qn)[p^\infty]|=p^{e_n'}$. Then, Lemma~\ref{doks} tells us that $e_n^\circ-e_{n-1}^\circ=e_n'-e_{n-1}'$ for $n\gg0$. As $e_n^\circ+e_n'=\ord_p(\Sha(E/K_n)[p^\infty])$, the result follows from Corollary~\ref{cor:sha}.
    \end{proof}

\printbibliography

\end{document}